\documentclass[preprint]{imsart}

\RequirePackage[OT1]{fontenc}
\RequirePackage{amsthm,amsmath,amssymb}
\RequirePackage{natbib}
\usepackage[hyperindex,breaklinks]{hyperref}
\usepackage{bbm}
\usepackage{amsthm}
\usepackage{enumerate}
\usepackage{paralist}

\numberwithin{equation}{section}
\theoremstyle{plain}
\newtheorem{theorem}{Theorem}[section]
\newtheorem{corollary}[theorem]{Corollary}
\newtheorem{lemma}[theorem]{Lemma}
\newtheorem{proposition}[theorem]{Proposition}
\theoremstyle{definition}
\newtheorem{example}[theorem]{Example}

\newtheorem{definition}[theorem]{Definition}
\newtheorem*{definition*}{Definition}
\theoremstyle{remark}
\newtheorem{remark}[theorem]{Remark}

\newcommand{\X}{\mathcal{X}}

\newcommand{\LL}{\mathcal{L}}
\newcommand{\B}{\mathcal{B}}

\newcommand{\F}{\mathcal{F}}

\newcommand{\M}{\mathcal{M}}

\newcommand{\A}{\mathcal{A}}
\newcommand{\E}{\mathbb{E}}
\newcommand{\PP}{\mathbb{P}}
\newcommand{\II}{\mathbbm{1}}

\newcommand{\field}[1]{\mathbb{#1}}
\newcommand{\R}{\field{R}}

\newcommand{\Propp}{\textbf{P}}
\newcommand{\dtv}{d_{\mathrm{TV}}}

\begin{document}

\begin{frontmatter}

\title{The convex distance inequality for dependent random variables, with applications to the stochastic travelling salesman and other problems}
\runtitle{The convex distance inequality for dependent random variables}

\begin{aug}
\author{\fnms{Daniel} \snm{Paulin} \ead[label=e1]{paulindani@gmail.com}}
\runauthor{D. Paulin}

\affiliation{National University of Singapore}

\address{
Department of Statistics and Applied Probability, National University of Singapore,\\
6 Science Drive 2, Singapore 117546, Republic of Singapore.\\
\printead{e1}}
\end{aug}



\begin{abstract}
We prove concentration inequalities for general functions of weakly dependent random variables satisfying the Dobrushin condition. In particular, we show Talagrand's convex distance inequality for this type of dependence. We apply our bounds to a version of the stochastic salesman problem, the Steiner tree problem, the total magnetisation of the Curie-Weiss model with external field, and exponential random graph models. Our proof uses the exchangeable pair method for proving concentration inequalities introduced by Chatterjee (2005). Another key ingredient of the proof is a subclass of $(a,b)$-self-bounding functions, introduced by Boucheron, Lugosi and Massart (2009).
\end{abstract}
\end{frontmatter}

\makeatletter{}\section{Introduction}
The theory of concentration of measure for functions of independent random variables has seen major development since the groundbreaking work of \cite{Talagrand1} (see the books \cite{Ledoux}, \cite{Rand}, and \cite{lugosi2013concentration}). These inequalities are very useful for obtaining non-asymptotic bounds on various quantities arising from models that are based on collections of independent random variables.

However, for many applications it may be difficult, if not impossible, to describe the model by means of a collection of independent random variables, whereas simpler descriptions based on dependent random variables may be readily available.  Such models arise, for example, in statistical physics, where certain distributions can be described as stationary distributions of appropriate Markov chains. Therefore, it is important to have concentration inequalities that are applicable beyond the independent setting.

In this paper, we will prove such inequalities for a certain type of dependence, namely for random variables satisfying the so-called the \emph{Dobrushin condition} (however, we believe that the methods presented here can also be adapted to other settings). This condition is satisfied, in particular, in certain statistical physical models when the temperature is sufficiently high,  and for sampling without replacement.

Concentration inequalities in the literature for random variables satisfying the Dobrushin condition can be found in the literature (see \cite{kulske2003concentration}, \cite{Martonstrongmixing},  \cite{Cth}, \cite{DGW}, \cite{Liming}, \cite{Chazottes08}, \cite{Ollivier1}, \cite{nywanglimingbernoulli}, \cite{nywang}).
Most of these results are variants of McDiarmid's bounded differences inequality, only taking into account the maximal deviations 
\[\sup_{x_1,\ldots,x_n,x_i'} |f(x_1,\ldots,x_i,\ldots,x_n)-f(x_1,\ldots,x_i',\ldots, x_n)|, \text{ for }1\le i\le n.\]

In order to get sharper bounds, it is natural to impose stronger conditions on the function $f$. In this article, we will do this by using the general formalism of \emph{$(a,b)$-self-bounding functions}, introduced for independent random variables by \cite{LugosiSelfBounding}. 

Our main contribution in this paper is the following. We will prove concentration inequalities for a slightly restricted subclass of $(a,b)$-self-bounding functions, which we call $(a,b)$-$*$-self-bounding (the reason for using the $*$, instead of a letter, is to make it clear that we have two parameters, $a$ and $b$). We show that our result implies a version of Talagrand's convex distance inequality for dependent random variables satisfying the Dobrushin condition.

Our approach in this paper is based on Stein's method of exchangeable pairs, as introduced in \cite{C2007}. Recently, other variants of Stein's method, size-biasing and zero-biasing, have been adapted to prove concentration inequalities, see \cite{GhoshGoldsteinPTRF}, and \cite{goldstein2013zerobiasconcentration}.

It is important to note that for certain types of dependence, such as uniform permutations (\cite{Talagrand1}) and Markov chains (\cite{Martoncontracting}, \cite{Samson}, \cite{Martonstrongmixing}, and \cite{Martoncoupling}) Talagrand's convex distance inequality was shown to hold. However, these approaches do not seem to easily generalise to dependent random variables satisfying the Dobrushin condition.

The rest of this article is organised as follows. 
In Section \ref{SecPreliminaries}, we will introduce the main definitions used in the article. In Section \ref{SecResults}, we present our main results. In Section \ref{SecApplications}, we discuss three applications, the stochastic salesman problem, the Steiner tree problem, and the total magnetisation of the Curie-Weiss model with external field. In Section \ref{SecPreliminaryResults} we prove some preliminary results, and in Section \ref{SecProofs}, we prove our main results. Finally, the Appendix includes a version of Talagrand's convex distance inequality for sampling without replacement.

\section{Preliminaries}\label{SecPreliminaries}
We start by introducing some notation. Let $X:=(X_1,\ldots,X_n)$ be a vector of random variables, where each $X_i$ takes values in a Polish space $\Lambda_i$, and, similarly, let $\Lambda:=\Lambda_1\times \Lambda_2\times \ldots \times \Lambda_n$, and let $\F$ be the Borel sigma algebra on $\Lambda$. 

For a vector $x$ in $\Lambda$, let $x_{-i}:=(x_1,\ldots,x_{i-1},x_{i+1},\ldots,x_n)$ be the vector created by dropping the $i$th coordinate, and set $\Lambda_{-i}:=\Lambda_1\times \ldots\times \Lambda_{i-1}\times \Lambda_{i+1}\times \ldots \times \Lambda_n$. The distribution of the random vector $X$ is denoted by $\mu$, and $(\Lambda,\F,\mu)$ is the probability space induced by $X$, that is, for $S\in \F$, $\mu(S)=\PP(X\in S)$.
The marginal distribution of $X_i$ given $X_{-i}=x_{-i}$ will be denoted by $\mu_i(\cdot|x_{-i})$.

We are going to use matrix norms. For an $n\times n$ matrix $A=(a_{ij})_{1\le i,j\le n}$, we denote its operator norms by $\|A\|_1$, $\|A\|_{\infty}$ and $\|A\|_2$, respectively. Note that, in particular, $\|A\|_1=\max_{1\le j\le n}\sum_{i=1}^n |a_{ij}|$ and $\|A\|_{\infty}=\max_{1\le i\le n}\sum_{j=1}^n |a_{ij}|$.

Let $g: \Lambda\to \R_+$ be a  non-negative function. We will be interested in the concentration properties of $g(X)$. We will denote its centered version by \[f(x):=g(x)-\E(g(X)).\]
The following definition of self-bounding functions is essentially that of \cite{LugosiSelfBounding}.
\begin{definition}\label{defsb} Let $a,b>0$.
A function $g:\Lambda\to \R_+$ is called \emph{$(a,b)$-self-bounding} if there exist measurable functions $g_i:\Lambda_{-i}\to \R$, $i=1,\ldots,n$,  such that for every $x\in \Lambda$, 
\begin{compactenum}[$(i)$]
\item $0\le g(x)-g_i(x_{-i})\le 1$ for $1\le i\le n$, and
\item $\sum_{i=1}^n (g(x)-g_i(x_{-i}))\le ag(x)+b$.
\end{compactenum}
A function $g:\Lambda\to \R$ is called \emph{weakly $(a,b)$-self-bounding} if for every $x\in \Lambda$,
\begin{enumerate}
\item[(\emph{ii'})] $\sum_{i=1}^n \left(g(x)-g_i(x_{-i})\right)^2\le a g(x)+b$;
\end{enumerate}
note that $(i)$ is not required in this case.
\end{definition}
\begin{remark}\label{selfboundingremark}
If $g$ is $(a,b)$-self-bounding, then it is also weakly $(a,b)$-self-bounding.
If $g$ is $(a,b)$-self-bounding, then we can always take the functions $g_i$ to be 
\begin{equation}\label{fimin}
g_i(x_{-i}):=\inf_{x_i'\in \Lambda_i}g(x_1,\ldots,x_{i-1},x_{i}',x_{i+1},\ldots,x_n).
\end{equation}
\end{remark}

We define $(a,b)$-$*$-self-bounding functions as follows.
\begin{definition}\label{defalphasb}
Let $a,b\ge 0$. A function $g:\Lambda\to \R$ is called \emph{$(a,b)$-$*$-self-bounding} if 
there exist measurable functions $\alpha_1,\ldots,\alpha_n: \Lambda\to \R$ such that
\begin{compactenum}[$(i)$]
\item $0\le \alpha_i(x)\le 1$,
\item for every $x,y\in \Lambda$, \[g(x)-g(y)\le \sum_{i: x_i\ne y_i}\alpha_i(x),\]
\item for every $x\in \Lambda$,
\[\sum_{i=1}^n \alpha_i(x)\le a g(x)+b.\]
\end{compactenum}
Similarly, a function $g:\Lambda\to \R$ is called \emph{weakly $(a,b)$-$*$-self-bounding} if there exists functions $\alpha_1,\ldots,\alpha_n:\Lambda\to \R_+$ such that $(ii)$ above holds, and 
\begin{enumerate}
\item[(\emph{iii'})] for every $x\in \Lambda$,
\[\sum_{i=1}^n \alpha_i(x)^2\le a g(x)+b;\]
\end{enumerate}
note that, again, $(i)$  is not required in this case.
\end{definition}

\begin{remark}
For each $a,b\ge 0$, the following relations hold.
\begin{center}
\begin{tabular}{c c c}
$(a,b)$-self-bounding & $\Rightarrow$ & weakly $(a,b)$-self-bounding\\
$\Uparrow$& &$\Uparrow$\\
$(a,b)$-$*$-self-bounding  & $\Rightarrow$ & weakly $(a,b)$-$*$-self-bounding
\end{tabular}
\end{center}
The reverse implications are false in general.
\end{remark}

The following definition allows us to quantify the dependence between the random variables.
\begin{definition}[Dobrushin's interdependence matrix]
Suppose $A=(a_{ij})$ is an $n\times n$ matrix with nonnegative entries and zeroes on the diagonal such that for every $i$, and every $x,y\in \Lambda$, 
\begin{equation}\label{Dobrushineq}
\dtv(\mu_i(\cdot | x_{-i}),\mu_i(\cdot | y_{-i}))\le \sum_{j\in [n]\setminus \{i\}} a_{ij}\II[x_j\ne y_j],
\end{equation}
where $\dtv$ denotes the total variational distance (see Section \ref{basicTV}), $[n]:=\{1,\ldots,n\}$, and $\mu_i(\cdot|x_{-i})=\PP(X_i\in \cdot |X_{-i}=x_{-i})$ denotes the marginal of $X_i$. We call such $A$ a \emph{Dobrushin interdependence matrix} for the random vector $X$ (or, equivalently, for the measure $\mu$). 
\end{definition}
\begin{remark}
The condition $\|A\|_1<1$ is commonly called the \emph{Dobrushin condition} in the literature. However, some authors use $\|A\|_2<1$ or $\|A\|_{\infty}<1$ instead.
The definition implicitly requires that $\mu_i(\cdot | x_{-i})$ exists for every $x_{-i}$. This may only be true in some of our applications in an almost sure sense. However, because we are going to assume that our random variables take values in a Polish space, we may use regular conditional probabilities, and change $\mu$ on a set of zero probability such that \eqref{Dobrushineq} becomes true everywhere, not just in an almost sure sense (see \cite{FadenRegular} for more details on the existence of regular conditional probabilities).
\end{remark}

\makeatletter{}\section{Main results}\label{SecResults}
In this section, we state our main results regarding concentration for $(a,b)$-$*$-self-bounding functions, and Talagrand's convex distance inequality. The results apply to weakly dependent random variables satisfying the Dobrushin condition.
\subsection{A new concentration inequality for $(a,b)$-$*$-self-bounding functions}\label{SecResultsDep}
Our main result is a bound on the moment generating function (mgf) of functions of random variables satisfying the Dobrushin condition.
\begin{theorem}
\label{thmsbdob}
Let $X=(X_1,\ldots,X_n)$ be a vector of random variables, taking values in $\Lambda$.
Let $A$ be a Dobrushin interdependence matrix for $X$, and suppose that $\|A\|_1<1$ and $\|A\|_{\infty}\le 1$.
Let $g:\Lambda\to \R$ be a non-negative measurable function such that $g(X)$ has finite mean, denoted by $\E(g)$. Let $a,b\ge 0$.
\begin{enumerate}
\item If $g$ is $(a,b)$-$*$-self-bounding, then for $0\le \theta \le (1-\|A\|_1)/a$,
\[\log \E \left[e^{\theta(g(X)-\E(g))}\right]\le \frac{(a\E(g)+b)\theta^2}{2(1-\|A\|_1-a\theta)}.\]
\item
If $g$ is weakly $(a,b)$-$*$-self-bounding, then for $0\le \theta\le (1-\|A\|_1)/(2a)$,
\begin{equation}\label{mgupdob}
\log \E\left[e^{\theta(g(X)-\E(g))} \right]\le \frac{(a\E(g)+b)\theta^2}{(1-\|A\|_1-2a\theta)}.
\end{equation}
\item
Suppose that $g$ is weakly $(a,b)$-$*$-self-bounding, and in addition, for every $x,x^*\in \Lambda$ differing only in one coordinate, $|g(x)-g(x^*)|\le 1$.
Then for $0\ge \theta\ge -\frac{1-\|A\|_1}{2a}$, the following inequality holds.
\begin{equation}\label{mddobres}
(\log m(\theta))'\ge -\left(e^{-\theta}-1\right)\frac{2}{1-\|A\|_1}\left(a\E(g)+b-\theta \frac{a(a\E(g)+b)}{2(1-\|A\|_1+2a\theta)}\right).
\end{equation}
\end{enumerate}
\end{theorem}
The proof of this is deferred to Section \ref{SecProofs}. As a corollary, we obtain concentration inequalities. For stating them, we will use a constant defined as follows. Let $a_c$ be the unique positive solution of
\begin{equation}\label{aeq}
\frac{\left(\exp(1/4a)-1\right)}{1/(4a)}=\frac{8}{5}.
\end{equation}
Note that $0.285<a_c<0.286$.
\begin{corollary}
\label{corsbdob}
Under the conditions of Theorem \ref{thmsbdob}, we have the following. \begin{enumerate}
\item If $g$ is $(a,b)$-$*$-self-bounding, then for all $t\ge 0$,
\[\PP[g(X)\ge \E(g) + t]\le \exp\left(-\frac{(1-\|A\|_1) t^2}{2(a\E(g)+b+at)} \right).\]
\item 
If $g$ is weakly $(a,b)$-$*$-self-bounding, then for all $t\ge 0$,
\[\PP[g(X)\ge \E(g)+t]\le \exp\left(-\frac{(1-\|A\|_1)t^2}{4(a\E(g)+b+at)}\right).\]
\item Suppose that $g$ is weakly $(a,b)$-$*$-self-bounding, and in addition, for every $x,x^*\in \Lambda$ differing only in one coordinate, $|g(x)-g(x^*)|\le 1$.
If $a\ge a_c(1-\|A\|_1)$, then for all $t\ge 0$, 
\[\PP[g(X)\le \E(g)-t]\le \exp\left(-\frac{(1-\|A\|_1)t^2}{8(a\E(g)+b)}\right),\]
while if $a\le a_c(1-\|A\|_1)$, then for all $t\ge 0$,
\[\PP[g(X)\le \E(g)-t]\le\exp\left(-\frac{t^2}{5(a\E(g)+b)/(1-\|A\|_1)+(2/3) t}\right).\]
\end{enumerate}
\end{corollary}
\subsection{The convex distance inequality for dependent random variables}\label{SecConvexdistance}
Recently, Talagrand's convex distance inequality was proven using the weakly self-bounding property in Section 2 of \cite{LugosiSelfBounding} (the original proof in \cite{Talagrand1} was based on mathematical induction). We are going to use similar ideas to prove a version of Talagrand's convex distance inequality based on Theorem \ref{thmsbdob} and, hence, applicable to dependent random variables  satisfying the Dobrushin condition. 

The result is stated in terms of  Talagrand's convex distance, which is defined as follows. For $c\in \R_+^n$, and $x,y \in \Lambda$, we define $d_{c}(x,y):=\sum_{i=1}^n c_i \II\left[x_i \ne y_i\right]$. For a point $x\in \Lambda$ and a set $S\subset \Lambda$, we let
$d_{c}(x,S):=\min_{y\in S}d_{c}(x,y)$ and
\begin{equation}\label{talagranddistance}
d_T(x,S):=\sup_{c\in \R_+^n, ||c||_2=1}d_{c}(x,S),
\end{equation}
which we call \emph{Talagrand's convex distance} between a point $x$ and a set $S$.
\begin{theorem}
\label{thmTalagrand}
Let $X:=(X_1,\ldots,X_n)$ be a vector of random variables, taking values in a Polish space $\Lambda=\Lambda_1\times \ldots \times \Lambda_n$, equipped with the Borel $\sigma$-algebra $\F$. Let $\mu$ be the probability measure on $\Lambda$ induced by $X$. Let $A$ be a Dobrushin interdependence matrix for $X$, and suppose that\, $\|A\|_1<1$ and\, $\|A\|_{\infty}\le 1$.
Then for any $S\in \F$,
\begin{equation}\label{Talagranddepeq}\E\left[e^{d_T(X,S)^2\cdot (1-\|A\|_1)/26.1}\right]\le \frac{1}{\mu(S)}.\end{equation}
\end{theorem}
\begin{remark}
Inequality \eqref{Talagranddepeq} is of the same form as Talagrand's original convex distance inequality in the independent case, but the latter holds with the constant $(1-\|A\|_1)/26.1$ being replaced by $1/4$. Our bound takes into account the strength of dependence between the random variables.
\end{remark}
The following corollary of the above result generalises the so-called ``method of non-uniformly bounded differences" to dependent random variables satisfying the Dobrushin condition.
\begin{corollary}
\label{nonunidep}
Let $X=(X_1,\ldots,X_n)$ be a vector of random variables, taking values in $\Lambda$, equipped with the Borel $\sigma$-algebra $\F$. Let $\mu$ be the probability measure on $\Lambda$ induced by $X$. Let $A$ be a Dobrushin interdependence matrix for $X$, and suppose that $\|A\|_1<1$ and $\|A\|_{\infty}\le 1$.
Let $g:\Lambda\to \R$ be a function satisfying that for some positive functions $c_1,\ldots,c_n :\Lambda\to \R_+$,
\begin{equation}\label{nonuccondeq}
g(x)-g(y)\le \sum_{i=1}^{n}c_i(x)\cdot \II[x_i\ne y_i]
\end{equation}
for every  $x=(x_1,\ldots,x_n)$, $y=(y_1,\ldots,y_n)$ in $\Lambda$,
and 
\begin{equation}\sum_{i=1}^{n} c_i^2(x)\le C
\end{equation}
uniformly for every $x$ in $\Lambda$. Then for any $t\ge 0$,
\begin{equation}
\PP(|g(X)-\mathbb{M}(g)|\ge t)\le 2\exp\left(\frac{-t^2\cdot (1-\|A\|_1)}{26.1 C}\right),
\end{equation}
where $\mathbb{M}(f)$ denotes the median of $g(X)$ (if the median is not unique, then the result holds for all of them).
\end{corollary}
\begin{proof}
The proof is along the same lines as the proof of Lemma 6.2.1 on page 122 of \cite{Steele}, except that the constant 4 is replaced by $26.1/(1-\|A\|_1)$.
\end{proof}

\makeatletter{}\section{Applications}\label{SecApplications}
In this section, we apply our results to a variant of the stochastic travelling salesmen problem, Steiner trees, the Curie-Weiss model, and exponential random graphs.
\subsection{Stochastic travelling salesman problem}\label{SecTSP}
One important and well studied problem in combinatoric optimisation is the travelling salesman problem (TSP). In the simplest, and most studied case, we are given $n$ points in the unit square $[0,1]^2$, and we are required to find the shortest tour, that is, to find the permutation $\sigma\in S_n$ ($S_n$ denoting the symmetric group) that minimises 
\[|x_{\sigma(1)}-x_{\sigma(2)}|+\ldots+|x_{\sigma(n)}-x_{\sigma(1)}|,\]
where $|x-y|$ denotes the Euclidean distance between $x$ and $y$.

Let us denote the length of the minimal tour by $T(x_1,\ldots,x_n)$. There has been much effort to find efficient algorithms to compute the minimal tour (in general, this is a difficult, NP complete problem, but there are fast algorithms that find a tour that is at most a fixed constant times worse than the optimal tour, see \cite{applegate2011traveling} for a recent book on this topic).

From a probabilistic point of view, it is of interest to look at the concentration properties of $T(X_1,\ldots,X_n)$, where $X_1,\ldots, X_n$ is a random sample from~$[0,1]^2$.
One of the classical applications of Talagrand's convex distance inequality is to show that, if $X_1,\ldots,X_n$ are i.i.d.\ uniformly distributed in $[0,1]^2$, then $T(X_1,\ldots,X_n)$ is very sharply concentrated around its median (or equivalently, its expected value), with typical deviations of order 1. We are going to study a modified version of the travelling salesman problem.
Let $\A:=\{a_1,\ldots,a_N\}$ be a fixed set of distinct points in $[0,1]^2$. Let $L(x,y): \A^2 \to \R$ be the \emph{cost function}, satisfying that for some constant $\mathcal{C}$, 
\begin{equation}\label{tspLxycond}
|x-y|\le L(x,y)\le \mathcal{C} |x-y| \text{ for every }x,y\in \A,
\end{equation}
where $|x-y|$ denotes the Euclidean distance of $x$ and $y$. Note that the cost function does not need to be a metric, and we do not even assume that it is symmetric. A non-symmetric cost function may be used to model the time taken for driving between two locations in a city that are at different elevation, since going uphill can take longer than going downhill.

For any set of distinct points $\{x_1,\ldots,x_n\}\in \mathcal{A}$, we let $T(x_1,\ldots,x_n)$ be the shortest tour through all the points, that is the minimum of the sum
\[L(x(\sigma(1)),x(\sigma(2))+\ldots+L(x(\sigma(n)),x(\sigma(1)))\]
for $\sigma\in S_n$. Since $T$ is invariant under the permutation of the points, we will also use the notation $T(\{x_1,\ldots,x_n\})$.

Assume that a set of $n$ distinct points are chosen from $\A$ according some distribution $\mu$ on all the subsets of size $n$ of $\A$.
Let
\begin{align*}
r_{n,1}(\mu)&:=
\sup_{\B\subset \A\atop |\B|=n-1}
\sup_{b\in \A\setminus \B}\frac{\mu(\B\cup b)}{\displaystyle\sum_{b'\in \A\setminus \B}\mu(\B\cup b')}\\
r_{n,2}(\mu)&:=\sup_{\B\subset \A \atop |\B|=n-2} \sup_{b,c,d\in \A\setminus \B}
\left|\frac{\mu(\B\cup b\cup d)}{\displaystyle\sum_{d'\in \A\setminus (\B\cup b)}\mu(\B\cup b\cup d')}-\frac{\mu(\B\cup c\cup d)}{\displaystyle\sum_{d'\in \A\setminus (\B\cup c)}\mu(\B\cup c\cup d')}\right|,
\end{align*}
and define the \emph{inhomogeneity coefficient} of this distribution $\mu$ as
\begin{equation}\label{rhomudef}
\rho_n(\mu):=n\left(r_{n,1}(\mu)+(N-n)\cdot r_{n,2}(\mu)\right).
\end{equation}
This coefficient is related to the distance of the distribution $\mu$ from the uniform distribution on all sets of size $n$,  corresponding to sampling without replacement. The following theorem is the main result of this section.

\begin{theorem}[Stochastic TSP for random subsets]\label{thmtspdep}
Let $\mathcal{X}$ be a random subset of size $n$ of $\A$, chosen according to a distribution $\mu$, with inhomogeneity coefficient $\rho_n(\mu)<1$. Then for any $t\ge 0$,
\begin{equation}\label{tspconceq}\mu(|T(\X)-\M(T)|\ge t)\le 4\exp\left(-\frac{t^2(1-\rho_n(\mu))}{1671\mathcal{C}^2}\right),\end{equation}
where $\M(T)$ denotes the median of $T$.
\end{theorem}
\begin{remark}
The inequality has the same form as the original result in the independent case (in that bound, the exponent is of the form $4\exp(-t^2/64)$).
\end{remark}
\begin{example}
Now we give a simple example of a distribution $\mu$ on $\A$, which we call \emph{weighted sampling without replacement}. Let $p$ be a probability distribution on $[N]$ satisfying that $p(i)$ is strictly positive for every $i\in [N]$. Let us choose a random subset $\X\subset \A$ as follows. Initially, $\X$ is empty. First, we pick an index from $[N]$ according to $p$, and put the set in $\A$ corresponding this index into $\X$. Then, we pick another index from $[N]$, according to $p$ conditioned on not choosing the first index. We obtain $\X$ by iterating this procedure $n$ times in total. If we have picked the indices $I_1,\ldots,I_k\in [N]$ in the first $k$ steps, then $\PP(k+1\text{th point is }i)=\frac{p(i)}{\sum_{j\in [N]\setminus\{I_1, \ldots, I_k\}}p(j)}$ (for $0\le k< n$). This means that for any $i_1,\ldots,i_n\in [N]$, we have
\begin{align*}&\PP(I_1=i_1,\ldots,I_n=i_n)\\
&\quad=\II[i_1,\ldots,i_n \text{ are disjoint }]\cdot p(i_1)\cdot \frac{p(i_2)}{\sum_{j\in [N]\setminus \{i_1\}}p_j}\cdot \ldots \cdot \frac{p(i_n)}{\sum_{j\in [N]\setminus \{i_1,\ldots,i_{n-1}\}}p_j}.
\end{align*}
Based on this, for a set of $n$ disjoint points $\{a_{i_1},\ldots, a_{i_n}\}\subset \A$, we define
$\mu(\{a_{i_1},\ldots, a_{i_n}\})$ by averaging over all the possible ways the random variables $I_1,\ldots, I_n$ can take values $i_1,\ldots,i_n$, that is,
\[\mu(\{a_{i_1},\ldots, a_{i_n}\}):=\frac{1}{n!}\sum_{j_1,\ldots, j_n} \PP(I_1=j_1,\ldots,I_n=j_n),\]
with the summation in $j_1,\ldots, j_n$ is taken over all $n!$ enumerations of $i_1,\ldots,i_n$. Note that this sampling scheme can be equivalently formulated using independent exponentially distributed random variables with parameters $p_1,\ldots,p_N$ (exponential clocks), where we choose the sets corresponding to the indices of the smallest $n$ such exponential variables (the first $n$ clocks that ring). 

Let $p_{\max}:=\max_{i\in [N]}p(i)$ and $p_{\min}:=\min_{i\in [N]}p(i)$, then an elementary computation shows that for the weighted sampling without replacement scheme,
\begin{equation}
\rho_n(\mu)\le \frac{1}{2}\left(p_{\max}/p_{\min}+\left(p_{\max}/p_{\min}\right)^2\right)\cdot \frac{n}{N-n},
\end{equation}
which is smaller than $1$ if $n<N/\big[1+\big(p_{\max}/p_{\min}+\left(p_{\max}/p_{\min}\right)^2\big)/2\big]$.

Sampling without replacement corresponds to the case when $p(i)=1/N$ for every $i\in [N]$. In this case, the condition of our theorem, $\rho_n(\mu)<1$, is satisfied if $n<N/2$. 
In this particular case, using a theorem of Talagrand, we can show that the convex distance inequality holds for any $n\le N$, which implies that Theorem \ref{thmtspdep} also holds for any $n\le N$. 
See the Appendix for more details.
\end{example}
Note that it does not seem to be possible to deduce Theorem \ref{thmtspdep} using the results of \cite{Samson}. In the special case when $X_1, \ldots, X_n$ are $n$ samples taken without replacement out of $N$ possibilities, the total variational distance of the distributions 
$\LL(X_l|X_1=x_1,\ldots,X_k=x_k)$ and $\LL(X_l|X_1=x_1,\ldots,X_{k-1}=x_{k-1},X_k=x_k')$ is greater than $1/N$ if $x_k\ne x_k'$. This means that the above diagonal elements of the mixing matrix are at greater than $1/N$, and the matrix created by taking the square root of every element has $L^2$ norm of
$\mathcal{O}(1+n/\sqrt{N})$. Therefore we need to have $n$ to be $O(\sqrt{N})$ to obtain concentration results that are only a constant times worse than in the independent case, whereas with our method, this is true for any $n<N/2$.
 
Now we turn to the proof of  Theorem \ref{thmtspdep}. The proof consists of two parts. Firstly, we compute the coefficients of the Dobrushin interdependence matrix and verify the Dobrushin condition. Secondly, we check that the function $T$ satisfies the conditions of Corollary \ref{nonunidep}.

The Dobrushin interdependence matrix is estimated in the following Lemma.
\begin{lemma}\label{tsplemma}
Let $\mu$ be a distribution on the subsets of size $n$ of $\A$. Let $X_1,\ldots,X_n$ be random variables taking values in $\A$, distributed as
\[\PP\left(X_1=a_{i_1},\ldots,X_n=a_{i_n}\right)=\frac{\mu(\{a_{i_1},\ldots,a_{i_n}\}) }{n!} \text{ for any distinct }i_1,\ldots,i_n\in [N].\]
Then there is a Dobrushin interdependence matrix for $X_1,\ldots, X_n$ such that
\[\|A\|_1,\|A\|_{\infty}\le \rho_n(\mu).\]
\end{lemma}
\begin{proof} Define the event $F_{n-1}(\B,b):=\{\{X_1,\ldots,X_{n-2}\}=\B, X_{n-1}=b\}$ for every $\B\subset \A, |\B|=n-2$ and $b\in \A\setminus \B$. By the definition of the Dobrushin interdependence matrix, using the triangle inequality for the total variational distance, we can set
\begin{align*}
a_{n(n-1)}&=\sup_{\B\subset \A, |\B|=n-2,\atop
b,c\in \A\setminus \B}\dtv\big(\LL(X_n| F_{n-1}(\B,b),
\LL(X_n| F_{n-1}(\B,c))\big)\\
&=\sup_{\B\subset \A, |\B|=n-2,\atop
b,c\in \A\setminus \B}\frac{1}{2}\sum_{d\in \A\setminus \B}\big|\PP(X_n=d|F_{n-1}(\B,b))-\PP(X_n=d|F_{n-1}(\B,c))\big|.
\end{align*}
This sum has two type of terms, the first type is when $d$ equals $b$ or $c$, and the second type is when $d$ equals something else in $\A\setminus \B$. Terms of the first type are less then equal to $r_{n,1}(\mu)$, and terms of the second type are bounded by $r_{n,2}(\mu)$, thus $a_{n(n-1)}\le \rho_n(\mu)/n$. Because of the symmetry of the distribution of $X_1, \ldots, X_n$, the same holds for every $a_{ij}$, thus the claim of the lemma follows.
\end{proof}

The following lemma will be used to verify the properties of the function $T$.
\begin{proposition}[Proposition 11.1 of \cite{Rand}]\label{spacefillingprop}
There is a constant $c>0$ such that, for any set of points $x_1,\ldots,x_n \in [0,1]^2$, there is a permutation $\sigma\in S_n$ satisfying 
\begin{equation}\label{sigmaceq}|x_{\sigma(1)}-x_{\sigma(2)}|^2 +\ldots + |x_{\sigma(n)}-x_{\sigma(1)}|^2\le c.\end{equation}
That is, there is a tour going trough all points such that the sum of the squares of the lengths of all edges in the tour is bounded by an absolute constant $c$. By the argument outlined in Problem 11.6 of \cite{Rand}, the above holds with $c=4$.
\end{proposition}

The following lemma summarises the properties of the function $T$ required for our proof.
\begin{lemma}\label{Ttsplemma}
For any $x, y\in  \A^n$, there are functions $\alpha_1, \ldots, \alpha_n: [0,1]^2\to \R_+$ such that we have
\begin{equation}\label{tspTpropeq}
T(x)-T(y)\le \sum_{i=1}^{n} \alpha_i(x) \II[x_i\ne y_i],
\end{equation}
and for any $x\in  \A^n$,
\begin{equation}\label{tspalphaeq}
\sum_{i=1}^{n} \alpha_i^2(x)\le 64\mathcal{C}^2,
\end{equation}
where $\mathcal{C}$ is as in \eqref{tspLxycond}.
\end{lemma}
\begin{proof}

For any $x_1,\ldots,x_n\in  \A$, let $\hat{\sigma}$ be the permutation in $S_n$ that 
satisfies \eqref{sigmaceq}. If there are several such permutations, we choose the one that is smallest in the ordering of permutations ranging from $(1,2,\ldots, n)$ to $(n,n-1,\ldots,1)$.
For $1\le i\le n$, define $\alpha_i(x_1,\ldots,x_n)$ as
\[\alpha_i(x_1,\ldots,x_n):=2[L(x_{\hat{\sigma}(i-1)},x_{\hat{\sigma}(i)})+L(x_{\hat{\sigma}(i)},x_{\hat{\sigma}(i+1)})],\]
with $i-1$ and $i+1$ taken in the modulo $n$ sense. With this choice, inequality \eqref{tspTpropeq} is proven on page 125 of \cite{Steele}, see also page 144 of \cite{Rand}. Inequality \eqref{tspalphaeq} follows from Proposition \ref{spacefillingprop}, and the condition $|x-y|\le L(x,y)\le \mathcal{C}|x-y|$.
\end{proof}

Now we are ready to prove our concentration result.
\begin{proof}[Proof of Theorem \ref{thmtspdep}]
We obtain \eqref{tspconceq} by applying Corollary \ref{nonunidep} to $T(X_1,\ldots,X_n)$, with $\|A\|_1\le \rho_n(\mu)$ and $C=64\mathcal{C}^2$.
\end{proof}

\subsection{Steiner trees}\label{SecSteiner}
Suppose that $H=\{x_1,\ldots,x_n\}$ is a set of $n$ distinct points on the unit square~$[0,1]^2$. Then the minimal spanning tree (MST) of $H$ is a connected graph with vertex set $H$ such that the sum of the edge length is minimal (in Euclidean distance). The \emph{minimal Steiner tree} of $H$ is the minimal spanning tree containing $H$ as a subset of its vertices. By the definition, the sum of the edge lengths of this is less than equal to the sum of the edge lengths of the minimal spanning tree, since we can also add vertices and edges to the graph (an example where they differ is the equilateral triangle, where the minimal Steiner tree adds the centre of mass of the triangle to the graph, thus reducing the total edge length). We denote the sum of the edge lengths of the minimal Steiner tree by $S(x_1,\ldots,x_n)$. Note that this is invariant to permutations of $x_1,\ldots,x_n$, thus we can equivalently denote it by $S(\{x_1,\ldots,x_n\})$.

This is a quantity of great practical importance, since it expresses the minimal amount of interconnect needed between the points $x_1,\ldots,x_n$. It has found numerous applications in circuit and network design. \cite{Thesteinertreeproblem} is a popular book on this subject.

From a probabilistic perspective, a problem of interest is to quantify the behaviour of $S(X_1,\ldots,X_n)$, where $X_1,\ldots,X_n$ are random variables that are i.i.d. uniformly distributed on $[0,1]^2$. \cite{Steele} has proven that the total length of the minimal Steiner tree, $S(X_1,\ldots,X_n)$, is sharply concentrated around its median, with typical deviations of order 1.

Here we study a modified version of this problem, when we choose a random subset of size $n$ from a set of points $\A:=\{a_1,\ldots,a_N\}$ in $[0,1]^2$. Let $\mu$ be a probability measure on such subsets, and denote its inhomogeneity coefficient defined in \eqref{rhomudef} by $\rho_n(\mu)$.
Using our version of Talagrand's convex distance inequality for dependent random variables, we obtain the following concentration bound.

\begin{theorem}[Minimal Steiner tree for random subsets]\label{Steinerthm}
Let $\mathcal{X}$ be a random subset of size $n$ of $\A$, chosen according to a distribution $\mu$, with inhomogeneity coefficient $\rho_n(\mu)<1$. Then for any $t\ge 0$,
\begin{equation}\label{steinerconceq}\PP(|S(\X)-\M(S)|\ge t)\le 4\exp\left(-\frac{t^2(1-\rho_n(\mu))}{520000}\right),\end{equation}
where $\M(S)$ denotes the median of $S$.
\end{theorem}

The proof consists, again, of two parts. First, we bound the Dobrushin interdependence matrix, then show that the function $S$ satisfies the conditions of our version of the method of non-uniformly bounded differences for dependent random variables (Corollary \ref{nonunidep}). The first part is proven in Lemma \ref{tsplemma}. For the second part, we are going to use the following lemma.

\begin{lemma}[\cite{Steele}, page 107, equation (5.26)]\label{elemma}
Let us denote the edge lengths of the minimum spanning tree for $x_1,\ldots,x_n\in [0,1]^2$ by $e_1,\ldots,e_{n-1}$. Then for some universal constant $c$,
\begin{equation}
e_1^2+\ldots+e_{n-1}^2\le c,
\end{equation}
in particular, we can choose $c=410$ (see page 108 of \cite{Steele}). If there are multiple minimal spanning trees, then this holds for each of them.
\end{lemma}

The conditions on $S$ are verified in the following lemma.
\begin{lemma}\label{Slemma}
For any $x_1,\ldots,x_n\in [0,1]^2$, denote $x=(x_1,\ldots,x_n)$, and for $1\le i\le n$, define $\alpha_{i}(x)$ as two times the length of the incurring edges in the minimal spanning tree of $x_1,\ldots,x_n$. Then for any $x,y\in ([0,1]^2)^n$, we have
\[S(x)-S(y)\le \sum_{i=1}^{n} \alpha_i(x)\cdot \II[x_i\ne y_i].\]
Moreover, for any $x\in ([0,1]^2)^n$, 
\[\sum_{i=1}^{n}\alpha_i^2(x)\le 19680.\]
\end{lemma}
\begin{proof}
The first claim is proven on pages 123-124 of \cite{Steele}. For the second claim, first notice that the vertices in the minimum spanning tree can have degree at most 6. Now for any 6 reals $z_1,\ldots,z_6$, we have $(z_1+\ldots+z_6)^2\le 6(z_1^2+\ldots+z_6^2)$, and every edge belongs to two vertex so it is counted twice, thus by Lemma \ref{elemma}, we have
\[\sum_{i=1}^{n}\alpha_i^2(x)\le 6\cdot 2^2 \cdot 2\sum_{i=1}^{n-1}e_{i}^2\le 19680.\]
\end{proof}
Now we are ready to prove our concentration result.
\begin{proof}[Proof of Theorem \ref{Steinerthm}]
Using Lemma \ref{tsplemma} and Lemma \ref{Slemma}, the statement of the theorem follows by applying Corollary \ref{nonunidep} with $\|A\|_1=\|A\|_{\infty}=\rho_n(\mu)$ and $C=19680$.
\end{proof}

\subsection{Curie-Weiss model}\label{SecCW}
The Curie-Weiss model of ferromagnetic interaction is the following. Consider the state space $\Lambda=\{-1,1\}^n$, and denote an element of the state space (a configuration) by $\sigma=(\sigma_1,\ldots,\sigma_n)$. Define the Hamiltonian for the system as
\[H(\sigma):=\left(\beta\frac{1}{n}\sum_{1\le i<j\le n}\sigma_i\sigma_j+h\sum_{i=1}^n\sigma_i\right),\]
and the probability density
\[p_{\beta}(\sigma):=\frac{\exp(\beta H(\sigma))}{Z(\beta,h)},\]
where $Z(\beta,h):=\sum_{\sigma\in \Lambda}\exp(\beta H(\sigma))$ is the normalizing constant.
The following proposition gives bounds on the  Dobrushin interdepence matrix for this model.
\begin{proposition}
For $\sigma$ as above, the Dobrushin interdependence matrix $A$ satisfies 
\[\|A\|_{1},\|A\|_{\infty},\|A\|_2< \beta.\]
\end{proposition}
\begin{proof}
We will now calculate the Dobrushin interdependence matrix for this system. Suppose first that $h=0$. Let $x$ and $y$ be two configurations, then
we want to bound
\[\dtv(\mu_i(\cdot |x_{-i}),\mu_i(\cdot |y_{-i}))\]
Since $\sigma_i$ can only take values $1$ or $-1$, so the total variation distance is simply
\[\dtv(\mu_i(\cdot |x_{-i}),\mu_i(\cdot |y_{-i}))=|\PP(\sigma_i=1|x_{-i})-\PP(\sigma_i=1|y_{-i})|.\]
Now by writing $m_i(x):=\frac{1}{n}\sum_{j: j\ne i} x_j$ and $m_i(y):=\frac{1}{n}\sum_{j: j\ne i} y_j$, we can write 
\[\PP(\sigma_i=1|x_{-i})=\frac{\exp(\beta m_i(x))}{\exp(\beta m_i(x))+\exp(-\beta m_i(x))},\]
so by denoting \begin{equation}\label{rdef}
r(t):=\frac{\exp(t)}{\exp(t)+\exp(-t)}=\frac{1}{1+\exp(-2t)},
\end{equation}
we can write
\[|\PP(\sigma_i=1|x_{-i})-\PP(\sigma_i=1|y_{-i})|=|r(\beta m_i(x))-r(\beta m_i(y))|.\]
Now it is easy to check that 
$|r'(t)|\le \frac{1}{2}$, and changing one spin in $x$ can change $m_i$ at most by $2/n$.
From this, we obtain a Dobrushin interdependence matrix $A$ with $a_{ij}=\frac{\beta}{n}$ for $i\ne j$.
For this $A$, it is easy to see that 
\[\|A\|_1=\|A\|_{\infty}=\|A\|_2=\beta\left(1-\frac{1}{n}\right)<\beta.\qedhere\]
\end{proof}

Thus for the high temperature case $0\le \beta<1$,  we can apply Corollary \ref{corsbdob} to obtain concentration inequalities.

In the case when writing the conditional probabilities for $h\ne 0$, one can show that in the above argument,  $r(t)$ in \eqref{rdef} gets replaced by $r(t,h):=\frac{\exp(t+h)}{\exp(t+h)+\exp(-t-h)}$. This function still satisfies that $|\frac{\partial}{\partial t}r(t,h)|\le 1/2$, thus $A$ as defined above is a Dobrushin interdependence matrix in this case as well.

Now we are going to show a concentration inequality for the average magnetization of the Curie-Weiss model. Let us denote the average magnetization by $m:=\frac{1}{n}\sum_{i=1}^n \sigma_i$. We have the following proposition.
\begin{proposition}\label{CWprop}
For the above model, when $0\le \beta<1$, and $h\ge 0$, we have
\begin{align*}
\PP(m(\sigma)&\ge \E(m(\sigma))+t) \le \exp\left(-\frac{n (1-\beta)t^2}{16(1-\tanh(h)+4/((1-\beta)\sqrt{n})}\right)\\
\PP(m(\sigma)&\le \E(m(\sigma))-t) \le \exp\left(-\frac{n (1-\beta) t^2}{4[1-\tanh(h)+4/((1-\beta)\sqrt{n})] +4t}\right).
\end{align*}
\end{proposition}
\begin{remark}
Since $1-\tanh(h)\le 2\exp(-2h)$  for $h\ge 0$, this proposition is better for large values of $h$ than what we could obtain from McDiarmid's bounded differences inequality (Theorem 4.3 of \cite{Cth}). That result uses only the Hamming Lipschitz property, and gives bounds of order $\exp(-n(1-\beta)t^2)$), which does not capture the fact that in such cases $\sigma_i$ and thus $m(\sigma)$ has small variance.
\end{remark}

\begin{proof}[Proof of Proposition \ref{CWprop} ]
Let $n_-(\sigma)=\sum_{i=1}^n \II[\sigma_i=-1]$ be the number of $-1$ spins, then
$m=\frac{n-2n_-}{n}$, and for $t\ge 0$,
\begin{align}
\label{eqmnm1}&\PP(m(\sigma)\ge \E(m(\sigma))+t)=\PP\left(n_-(\sigma)\le \E (n_-(\sigma))-\frac{n}{2}t\right),\\
\label{eqmnm2}&\PP(m(\sigma)\le \E(m(\sigma))-t)=\PP\left(n_-(\sigma)\ge \E (n_-(\sigma))+\frac{n}{2}t\right).
\end{align}
Here $n_-(\sigma)$ is a sum of non-negative variables, so one can easily see that it is~$(1,0)$-$*$-self-bounding, and thus, by Theorem \ref{thmsbdob}, we have for every $t\ge 0$,
\begin{align}
\label{eqnm1}&\PP(n_-(\sigma)\ge \E(n_-(\sigma))+t)\le \exp\left(-\frac{(1-\beta)t^2}{2\E (n_-(\sigma))+2t}\right)\\ 
\label{eqnm2}&\PP(n_-(\sigma)\le \E(n_-(\sigma))-t)\le \exp\left(-\frac{(1-\beta)t^2}{8\E (n_-(\sigma))}\right).
\end{align}

In order to apply this bound, we will need to estimate $\E (n_-(\sigma))=n(1-\E(m))/2$. For this, we are going to use Proposition 1.3 of \cite{C2007}, stating that for any $t\ge 0$,
\begin{equation}\label{CWmconceq}
\PP\left(m(\sigma)-\tanh(\beta m(\sigma)+h)\ge \frac{\beta}{n}+\frac{t}{\sqrt{n}}\right)
\le \exp(-t^2/(4+4\beta)),
\end{equation}
and the same bound holds for the lower tail as well. Here we have replaced $\beta h$ with $h$ in the equation of Proposition 1.3 because of the different definition of the Hamiltonian of the model. 
Now for $0\le \beta<1$, the equation $m=\tanh(\beta m+h)$ admits a unique solution in $m$, which we denote by $m^*(h)$.

For $0\le \beta\le 1$, \eqref{CWmconceq} can be further bounded by $\exp(-nt^2/8)$, moreover, for any $x\ge 0$,
$\PP(|m(\sigma)-m^*|\ge x/(1-\beta))\le \PP(|m(\sigma)-\tanh(\beta m(\sigma)+h)|\ge x)$, and thus for any $t\ge 0$,
\[\PP\left((m(\sigma)-m^*) \ge \left(\frac{1}{1-\beta}\right)\cdot \left(\frac{1}{n}+\frac{t}{\sqrt{n}}\right)\right)\le \exp(-t^2/8),\] 
and the same inequality holds for the lower tail as well, but with $m(\sigma)-m^*$ replaced by $m^*-m(\sigma)$.
From this, using integration by parts, we obtain that
\[\E((m(\sigma)-m^*)_+),\E((m(\sigma)-m^*)_-)\le \frac{1}{1-\beta}\cdot \frac{1}{n}+\frac{1}{1-\beta}\cdot \frac{1}{\sqrt{n}}\cdot \sqrt{2\pi}\le \frac{4}{(1-\beta)\sqrt{n}},\]
implying that $|\E(m(\sigma))-m^*|\le 4/((1-\beta)\sqrt{n})$. Now it is easy to see that for~$h\ge 0$, we have 
$m^*(h)\ge \tanh(h)$, and thus $\E(m(\sigma))\ge \tanh(h)-4/((1-\beta)\sqrt{n})$ and 
\[\E (n_-(\sigma))\le n(1+4/((1-\beta)\sqrt{n})-\tanh(h))/2.\]
Now the results follow by combining this with equations \eqref{eqmnm1}, \eqref{eqmnm2}, \eqref{eqnm1} and \eqref{eqnm2}.
\end{proof}

\subsection{Exponential random graphs}\label{SecExp}
Exponential random graph models are increasingly popular for modelling network data (see \cite{ChatterjeeExpRandGraphs}). For a graph with $n$ vertices, the edges are distributed according to a probability distribution of the form
\begin{equation}\label{expgrapheq}
p_{\beta}(G):=\exp\left(\sum_{i=1}^{k}\beta_i T_i(G)-\psi(\beta)\right),
\end{equation}
where $\beta=(\beta_1,\ldots,\beta_k)$ is a vector of real parameters, and $T_1,\ldots, T_k$ are functions on the space of the graphs ($T_1$ is usually the number of edges, while the rest can be the number of triangles, cycles, etc.\ ), and $\psi(\beta)$ is the normalising constant.

The simplest special case of this model is the Erd\H{o}s-R\'{e}nyi graph. Let $E$ be the number of edges of the graph, and let $0<p<1$ be a parameter, then in this case,
\[p_{\beta}(G):=p^{E} (1-p)^{n(n-1)/2-E}=\exp\left(\log\left(\frac{p}{1-p}\right)E +\log(1-p)n(n-1)/2\right).\]
In this case, the edges are i.i.d.\ random variables distributed according to the Bernoulli distribution with parameter $p$.

A more complex model, which was analysed in \cite{ChatterjeeExpRandGraphs}, has the distribution
\[p_{\beta_1,\beta_2}(G)=\exp\left(2\beta_1 E + \frac{6\beta_2}{n}\Delta-n^2 \psi_n(\beta_1,\beta_2)\right),\]
where $E$ denotes the number of edges, $\Delta$ denotes the number of triangles, and $\psi_n(\beta_1,\beta_2)$ is the normalising constant. Note that in this case, the edges are no longer independent, because the number of triangles introduces a form of dependence into the model.

In general, for any model of the type \eqref{expgrapheq}, there is a certain set $\mathcal{D}\subset \R^k$ of non-zero volume such that when the parameters $\beta\in \mathcal{D}$, the edges, as random variables, satisfy the Dobrushin condition (that is, there is an interdependence matrix such that $\|A\|_1<1$ and $\|A\|_{\infty}<1$). This fact can be shown by a simple continuity argument, since the random variables are independent when $\beta=0$. The set $\mathcal{D}$ is analogous to the high-temperature phase of statistical physical models.

The following theorem, based on our new concentration inequality for $(a,b)$-*-self-bounding functions, establishes concentration inequalities for subgraph counts in exponential random graph models in the high temperature phase.

\begin{theorem}[Subgraph counts in exponential random graphs]\label{Subgraphexpthm}\hspace{5mm}\\
Let $\Lambda:=\{0,1\}^{n(n-1)/2}$, and let $X:=(X_{ij})_{1\le i<j\le n}$ be the edges of an exponential random graph, taking values in $\Lambda$, distributed according to 
$p_{\beta}$, as defined by \eqref{expgrapheq}. Suppose that $\beta\in \mathcal{D}$.

Let $S$ be a fixed graph with $n_S$ vertices and $e_S$ edges. Let $N_S$ denote the number of copies of $S$ in our exponential random graph, then for any $t\ge 0$,
\begin{align}
\PP(N_S-\E(N_S) \ge t)&\le \exp\left(\frac{(1-\|A\|_1)t^2}{2{n-2 \choose n_S-2} e_S\cdot (\E(N_S)+t)}\right),\\
\PP(N_S-\E(N_S) \le -t)&\le \exp\left(\frac{(1-\|A\|_1)t^2}{8{n-2 \choose n_S-2} e_S \cdot \E(N_S)}\right).
\end{align}
\end{theorem}
\begin{remark}
By the number of copies of $S$, we mean the number of subsets of size $n_S$ of the set of $n$ vertices of our graph such that the corresponding subgraph contains $S$. 
A of similar concentration inequality can be shown to hold for the maximal degree among all the vertices (see Example 6.13 of \cite{lugosi2013concentration}), which can be shown to be $(1,0)$-*-self-bounding. Our results are sharper than what we could obtain using Theorem 4.3 of \cite{Cth} (McDiarmid's bound differences inequality for dependent random variables satisfying the Dobrushin condition).
\end{remark}
\begin{proof}[Proof of Theorem \ref{Subgraphexpthm}]
The proof is based on the *-self-bounding property of $N_S$. If we add an edge to $X$, then $N_S$ will increase, or stay the same, while if we erase an edge from $X$, then $N_S$ will decrease, or stay the same. For $x\in \Lambda$, $1\le i<j\le n$, let $\alpha_{i,j}(x)$ be the number of copies of $S$ in $x$ that contain the edge $(i,j)$. Then $0\le \alpha_{i,j}(x)\le {n-2 \choose n_S-2}$, and we can see that for any $x,y \in \Lambda$,
\[N_S(x)-N_S(y)\le \sum_{1\le i<j\le n}\alpha_{i,j}(x) \II[x_{ij}\ne y_{ij}].\]
Moreover, since $S$ contains $e_S$ edges, we have
\[\sum_{1\le i<j\le n}\alpha_{i,j}(x)\le e_S N_S(x).\]
This means that $N_S(x)/{n-2 \choose n_S-2}$ is $(e_S, 0)$-*-self-bounding, and the results follow by Corollary \ref{corsbdob}.
\end{proof}

\makeatletter{}\section{Preliminary results}\label{SecPreliminaryResults}
In this section, we will prove some preliminary results needed for proving our main results from Section \ref{SecResults}. First, we prove a lemma about the total variational distance. After this, review the basics of the concentration inequalities by Stein's method of exchangeable pairs approach. Finally, we prove some lemmas about bounding moment generating functions.
\subsection{Basic properties of the total variational distance}\label{basicTV}
The total variational distance of two probability distributions $\mu_1$ and $\mu_2$ defined on the same measurable space $(\X,\F)$ is defined as 
\begin{equation}\label{dTVdefeq}\dtv(\mu_1,\mu_2)=\sup_{S\in \F}|\mu_1(S)-\mu_2(S)|.\end{equation}
The following lemma proposes a coupling related to the total variational distance that we are going to use.
\begin{lemma}\label{dtvlemma}
Let $\mu_1$ and $\mu_2$ be two probability measures on a Polish space~$(\X,\F)$.
Then for any fixed $q$ with $\dtv(\mu_1,\mu_2)\le q\le 1$, 
we can define a coupling of independent random variables $\chi, B, C, D$ such that $\chi$ has Bernoulli distribution with parameter $q$, and the random variables
\begin{equation}X:=(1-\chi)B+\chi C,\hspace{5mm} Y:=(1-\chi)B+\chi D\end{equation}
satisfy that $X\sim \mu_1$, $Y\sim \mu_2$.
\end{lemma}
\begin{proof}
The proof is similar to Problem 7.11.16 of \cite{Grimmettex}.  We define the measure $\mu_{12}(\cdot )$ on $(\X,\F)$ as $\mu_{12}(S)=\frac{\mu_1(S)+\mu_2(S)}{2}$. Then $\mu_1$ and $\mu_2$ are both absolutely continuous with respect to $\mu_{12}$, thus we can define the Radon-Nikodym derivatives $f(x):=\frac{d \mu_1}{d \mu_{12}}(x)$ and $g(x):=\frac{d \mu_2}{d \mu_{12}}(x)$ for almost every $x\in \Omega$. 

The density of random variables $B$, $C$ and $D$ with respect to $\mu_{12}$ can be defined in terms of $f(x)$ and $g(x)$ as follows. Let us define $h:\X \to \R$ as $h(x)=\min(f(x),g(x))$, and let $p:=\dtv(\mu_1,\mu_2)$. For any $S\in \F$, we let
\begin{eqnarray*}
\mu_{B}(S)&:=&\int_{x\in S}\frac{h(x)}{1-p} d \mu_{12}(x),\\
\mu_{C}(S)&:=&\int_{x\in S}\left(h(x)\frac{q-1}{1-p}+f(x)\right)\frac{1}{q} d\mu_{12}(x),\\
\mu_{D}(S)&:=&\int_{x\in S}\left(h(x)\frac{q-1}{1-p}+g(x)\right)\frac{1}{q}d \mu_{12}(x),
\end{eqnarray*}
and we set $\chi\sim \mathrm{Bernoulli}(q), B\sim \mu_B, C\sim \mu_C, D\sim \mu_D$ be independent random variables. With this choice, it is straightforward to check that the conditions of the lemma are satisfied.
\end{proof}

\subsection{Concentration by Stein's method of exchangeable pairs}
Let $f:\X\to \R$, where $\X$ is a Polish space, and $X$ is a random variable taking values in $\X$. We are interested in the concentration properties of $f(X)$. Suppose that $\E (f(X))=0$. Let  $(X,X')$ be an exchangeable pair, $m(\theta):=\E (e^{\theta f(X)})$. Suppose that $F(x,y):\X^2\to \R$ is an antisymmetric function satisfying 
\begin{equation}\label{eqFXXp}
\E(F(X,X')|X)=f(X). 
\end{equation}
Then for any $\theta\in \R$,
\begin{align}
\nonumber m'(\theta)&=\E(f(X)e^{\theta f(X)})=\E(F(X,X')e^{\theta f(X)})=-\E(F(X,X')e^{\theta f(X')})\\
&=\E\left(F(X,X')\frac{e^{\theta f(X)}-e^{\theta f(X')}}{2}\right).\label{eqmdth0}
\end{align}
By \cite{Cth}, this can be further bounded by 
\[\E\left(\frac{1}{2}|F(X,X')||f(X)-f(X')| e^{\theta f(X)}\right),\]
and conditions on $\Delta(X):=\frac{1}{2}\E\left(\left.|F(X,X')||f(X)-f(X')| \right|X\right)$ determine the concentration properties of $f(X)$.

In this paper, we are also going to use \eqref{eqmdth0}, but instead of taking absolute value, we consider positive and negative parts.

In order to apply the approach for some function $f$, we need to find the antisymmetric function $F(x,y)$ such that \eqref{eqFXXp} is satisfied. Chapter 4 of \cite{Cth} finds such an antisymmetric function by a method using a Markov chain, we give a summary below.

An exchangeable pair $(X,X')$ automatically defines a reversible Markov kernel $P$ as
\begin{equation}\label{eq41}
Pf(x):=E(f(X')|X=x),
\end{equation}
where $f$ is any function such that $\E |f(X)|<\infty$.

Let $\{X(k)\}_{k\ge 0}$ and $\{X'(k)\}_{k\ge 0}$ be two chains with Markov kernel $P$, having arbitrary initial values, and coupled according to some coupling scheme which satisfies the following property.
\begin{description}
\item \Propp\/ For every initial value $(x,y)$ of the joint chain $\{X(k)\}_{k\ge 0}, \{X'(k)\}_{k\ge 0}$ , and every $k$, the marginal distribution of $X(k)$ depends only on $x$ and the marginal distribution of $X'(k)$ depends only on $y$.
\end{description}

Under this assumption, the following lemma holds.
\begin{lemma}[Lemma 4.2 of \cite{Cth}]\label{lemma42}
Suppose the chains $\{X(k)\}$ and $\{X'(k)\}$ satisfy the property \Propp \/ described above. Let $f: \X \to \R$ be a function such that $\E f(X)=0$. Suppose there exists a finite constant $L$ such that for every $(x,y)\in \X^2$,
\begin{equation}\label{eq44}
\sum_{k=0}^{\infty}|\E(f(X(k))-f(X'(k))|X(0)=x,X'(0)=y)|\le L.
\end{equation}
Then, the function $F$, defined as
\[F(x,y):= \sum_{k=0}^{\infty}\E (f(X(k))-f(X'(k))|X(0)=x,X'(0)=y),\]
satisfies $F(X,X')=-F(X',X)$ and $\E(F(X,X')|X)=f(X)$.
\end{lemma}

\subsection{Additional lemmas}
The following lemma proves concentration in the case when $\Delta(X)$ is not bounded almost surely, but itself is concentrated (a reformulation of Lemma 11 of \cite{massart2000constants}). Since the proof is short, we include it for completeness (it is based on part of the proof of Theorem 3.13 of \cite{Cth}).
\begin{lemma}\label{lemmabrut}
Let $m(\theta)=\E(e^{\theta f(X)})$. For any random variable $V$, and any $L>0$, we have for every $\theta \in \R$,
\[\E(e^{\theta f(X)} V)\le  L^{-1}\log \E(e^{LV}) m(\theta)+L^{-1}\theta m'(\theta)-L^{-1}m(\theta)\log(m(\theta)),\]
if the expectations on both sides exist.
\end{lemma}

\begin{proof}
Let $u(X):=\frac{e^{\theta f(X)}}{m(\theta)}$.  Let $A, B\ge 0$ be two random variables with finite variance and $\E(A)=1$, then 
\[\E(A \log (B))\le \log( \E(A B)),\]
which can be shown by changing the measure and applying Jensen's inequality. Using this, we have
\begin{eqnarray*}
\E(e^{\theta f(X)} V)&=& L^{-1} m(\theta)\E\left( u(X) \left(\log\frac{e^{LV}}{u(X)}+\log u(X)\right)\right)\\
&\le&  L^{-1}\log \E(e^{LV}) m(\theta)+L^{-1} \E\left(e^{\theta f(X)}\log u(X)\right),
\end{eqnarray*}
here we applied our previous inequality with $A=u(X)$ and $B=\frac{e^{LV}}{u(X)}$.
Now using the fact that $\log (u(X))=\theta f(X) - \log (m(\theta))$, we obtain the result.
\end{proof}
We will use the following well known result many times in our proofs.
\begin{lemma}\label{mdlemma}
Let $W$ be a centered random variable with moment generating function $m(\theta)$. Let $C,D\ge 0$, suppose that $m(\theta)$ is finite, and continuously differentiable in $[0,1/C)$, and satisfies
\[m'(\theta)\le C\theta m'(\theta)+D\theta m(\theta).\]
Then for $0\le \theta< 1/C$,
\begin{equation}\log(m(\theta))\le \frac{D\theta^2}{2(1-C\theta)},\end{equation}
and for every $t\ge 0$,
\begin{equation}\PP(W\ge t)\le \exp\left(-\frac{t^2}{2(D+Ct)}\right).\end{equation}
\end{lemma}
\begin{proof}
By rearranging, we have
\begin{align*}
(1-C\theta)m'(\theta)&\le D\theta m(\theta)\\
\log(m(\theta))'&\le \frac{D\theta}{1-C\theta}\\
\log(m(\theta))&\le \int_{x=0}^{\theta}\frac{Dx}{1-Cx}=-\frac{D\theta}{C}-\frac{D\log(1-C\theta)}{C^2}\le \frac{D\theta^2}{2(1-C\theta)},
\end{align*}
using the fact that for $0\le z\le 1$, $-z-\log(1-z)\le \frac{z^2}{2(1-z)}$. We obtain the tail bound by applying Markov's inequality for $\theta=\frac{t}{D+Ct}$.
\end{proof}

\makeatletter{}\section{Proofs of the main results}\label{SecProofs}
In this section, we are going to prove our main result, Theorem \ref{thmsbdob} and Corollary \ref{corsbdob}. The theorem concerns dependent random variables, and we need to introduce a certain amount of notation to handle them, making the proof rather technical. In order to help the reader in digesting this proof, we are going to prove the theorem first in the independent case, where we are free of the notational burden required for dependent random variables.

Before starting the proof in the independent case, we introduce some notation and two lemmas that are going to be used in both the independent and the dependent cases.

Let $X=(X_1,\ldots,X_n)$ be an vector of random variables taking value in $\Lambda$. Let $f:\Lambda\to \R$ be the centered version of $g$, defined as
\begin{equation}
f(x)=g(x)-\E(g(X)) \text{ for every }x\in \Lambda.
\end{equation}
Let $\alpha_1,\ldots,\alpha_n:\Lambda\to \R_+$ be functions such that for any $x,y\in \Lambda$,
\begin{equation}\label{alphaas}
f(x)-f(y)\le \sum_{i=1}^n \II[x_i\neq y_i] \alpha_i(x);
\end{equation}
let $\alpha(x):=(\alpha_1(x),\ldots,\alpha_n(x))$. Note that at this point we do not yet make any specific self-bounding type assumptions on $\alpha(x)$.

Let $I$ be uniformly distributed in $[n]$. Suppose that $(X,X')$ is an exchangeable pair, such that $X_i=X_i'$ for every $i\in [n]\setminus\{I\}$. Suppose that for $k\ge 0$, $X(k)$ and $X'(k)$ are Markov chains with kernel defined as in \eqref{eq41}, satisfying Property \/\Propp\/ and \eqref{eq44}. For $k\ge 0$, define the random vector $L(k) \in \R_+^n$ as 
\[L_i(k):=\II[X_{i}(k)\ne X'_{i}(k)]\text{ for }1\le i\le n.\] 
The following two lemmas bound the moment generating function of $f$ in function of the vectors $L(k)$ and $\alpha(x)$.
\begin{lemma}\label{lemmadob1}
Under the above assumptions, for $\theta>0$, if $m(\theta)<\infty$, then we have
\[m'(\theta)\le \E\left(\sum_{k=0}^{\infty} \left<L(k), \alpha(X(k))\right> \alpha_I(X) \theta e^{\theta f(X)}\right).\]
\end{lemma}
\begin{proof}
Note that
\begin{align*}
m'(\theta)&=\E(f(X) e^{\theta f(X)})\\
&=\E\left(F(X,X')e^{\theta f(X)}\right)=\frac{1}{2}\E\left(F(X,X')(e^{\theta f(X)}-e^{\theta f(X')}\right) \\
&\le \E\left((F(X,X'))_+(e^{\theta f(X)}-e^{\theta f(X')})_+\right)\\
&=\E\left((F(X,X'))_+(1-e^{-\theta (f(X)-f(X'))_+})e^{\theta f(X)}\right)\\
&\le \E\left((F(X,X'))_+(f(X)-f(X'))_+ \theta e^{\theta f(X)}\right)\\
&\le \E\left(\sum_{k=0}^{\infty} \left(f(X(k))-f(X'(k))\right)_+\left(f(X)-f(X')\right)_+ \theta e^{\theta f(X)}\right).
\end{align*}
Using \eqref{alphaas}, we have 
\[\left(f(X)-f(X')\right)_+\le \alpha_I(X),\text{ and }\left(f(X(k))-f(X'(k))\right)_+\le \left<L(k), \alpha(X(k))\right>,\]
thus the result follows.
\end{proof}

\begin{lemma}\label{lemmadob2}
Under the above assumptions, for $\theta<0$, if $m(\theta)<\infty$, and in addition, $f(X)-f(X')\le 1$ almost surely, then
\[m'(\theta)\ge - \sum_{k=0}^{\infty} \E\left(\left(e^{-\theta}-1\right) e^{\theta f(X)} \left<L(k),\alpha(X(k))\right> \alpha_I \right).\]
\end{lemma}
\begin{proof}
Note that
\begin{align*}
m'(\theta)&=\frac{1}{2}\E\left(F(X,X')\left(e^{\theta f(X)}-e^{\theta f(X')}\right)\right)\\
&\ge-\E\left((F(X,X'))_+\left(e^{\theta f(X)}-e^{\theta f(X')}\right)_-\right)\\
&\ge -\E\left((F(X,X'))_+\left(e^{\theta f(X')}-e^{\theta f(X)}\right)_+\right) \\
&\ge-\E\left((F(X,X'))_+\left(e^{\theta (f(X')-f(X))}-1\right)_+e^{\theta f(X)}\right)\\
&= -\E\left((F(X,X'))_+\left(e^{-\theta (f(X)-f(X'))_+}-1\right)e^{\theta f(X)}\right).
\end{align*}
Since $\theta<0$, and $\left(e^{(-\theta) x}-1\right)/x$ is a monotone function in $x$ for $x\ge 0$, using $0\le (f(X)-f(X'))_+\le 1$, we obtain
\[\left(e^{-\theta (f(X)-f(X'))_+}-1\right)\le (f(X)-f(X'))_+  \left(e^{-\theta}-1\right).\]
Now applying \eqref{alphaas} proves the result.
\end{proof}

\subsection{Independent case}
In this section, we are going to prove Theorem \ref{thmsbdob} and Corollary \ref{corsbdob} under the additional assumption that $X=(X_1,\ldots, X_n)$ is a vector independent random variables. First, we are going to construct a valid coupling of $(X(k), X'(k))_{k\ge 0}$, satisfying Property \/\Propp\/ and \eqref{eq44}. After this, we will use Lemma \ref{lemmadob1} and \ref{lemmadob1} to obtain the mgf bounds of Theorem \ref{thmsbdob}.

The construction of $(X(k), X'(k))_{k\ge 0}$ is the same as in Example on page 73 of \cite{Cth}, sketched here for the sake of completeness. This is a version of the Glauber dynamics. First, we set $X(0)=x$, and $X'(0)=y$ for some $x,y \in \Lambda$. Then we let $I(1), I(2), \ldots$ be independent random variables uniformly distributed on $[n]$, and $X^*(1), X^*(2), \ldots$ be independent copies of $X$. Then in the first step, we define the vectors $X(1)$ and $X'(1)$ as equal to $X(0)$, and $X'(0)$, respectively, except in coordinate $I(1)$, where we set $X_{I(1)}(1)=X'_{I(1)}(1)=X_{I(1)}^*(1)$. We define $X(k), X'(k)$ in the same way, by starting from $X(k-1), X'(k-1)$, and changing their coordinate $I(k)$ to $X_{I(k)}^*(k)$. This coupling has shown to satisfy  Property \/\Propp\/ and \eqref{eq44} in \cite{Cth} (via the coupon collector's problem). Finally, we note that $X'$ is defined as one step in the dynamics, that is, we let $X^*$ be an independent copy of $X$, $I$ be uniformly distributed on $[n]$, independently of $X$ and $X^*$, and $X'$ equals to $X$ except in coordinate $I$, where it equals $X^*_{I}$.

Now we are ready to prove Theorem \ref{thmsbdob} and Corollary \ref{corsbdob} under the independence assumption.
\begin{proof}[Proof of Part 1 of Theorem \ref{thmsbdob} and Corollary \ref{corsbdob} assuming independence]\hspace{10mm}\\
By Lemma \ref{lemmadob1}, using the fact that $f$ is bounded under our assumptions, we have that for $\theta>0$,
\[m'(\theta)\le \sum_{k=0}^{\infty} \E\left(\theta e^{\theta f(X)} \cdot \sum_{i=1}^{n} \alpha_i(X(k)) \alpha_i(X) \II[i\notin I(1),\ldots,I(k)] \right)\]
Now by our assumption, $\alpha_i(X(k))\le 1$, and using that $g$ is (a,b)-*-self-bounding,
\begin{align*}
&m'(\theta)\le \sum_{k=0}^{\infty}  \E\left(\theta e^{\theta f(X)} \cdot  \frac{1}{n}\sum_{i=1}^{n} \alpha_i(X) \II[i\notin I(1),\ldots,I(k)] \right)\\
&\le 
\E\left(\theta e^{\theta f(X)} \cdot  \frac{1}{n}\sum_{i=1}^{n} \alpha_i(X) \sum_{k=0}^{\infty} \left(1-\frac{1}{n}\right)^k \right)\\\
&\le \E\left(\theta e^{\theta f(X)} (a g(X)+b)\right)=\E\left(\theta e^{\theta f(X)} (af(X)+(a\E g(X)+b))\right)\\
&\le  \theta a m'(\theta)+\theta \left(a\E g(X) + b\right)m(\theta).
\end{align*}
The mgf bound now follows by rearrangement and integration, and applying Lemma \ref{mdlemma} proves the concentration bound of Corollary \ref{corsbdob}.
\end{proof}

\begin{proof}[Proof of Part 2 of Theorem \ref{thmsbdob} and Corollary \ref{corsbdob} assuming independence]\hspace{10mm}\\
By Lemma \ref{lemmadob1}, we have for $\theta>0$
\begin{equation}\label{mptheta1}
m'(\theta)\le \sum_{k=0}^{\infty} \E\left(\theta e^{\theta f(X)} \cdot  \frac{1}{n}\sum_{i=1}^{n} \alpha_i(X(k)) \alpha_i(X) \II[i\notin I(1),\ldots,I(k)] \right).
\end{equation}
Now by the fact that $g$ is weakly $(a,b)$-*-self-bounding, we have
\[
\sum_{i=1}^n \alpha_i(X)^2\le a g(X)+b,\quad \text{ and }\quad \sum_{i=1}^n \alpha_i(X(k))^2\le a g(X(k))+b.
\]
We will use the conditional version of the Cauchy-Schwarz inequality: if $A_i,B_i$ are random variables for $1\le i\le n$, then
\begin{align*}\E(A_i B_i|X)&\le \left(\E(A_i^2 |X)\right)^{1/2} \cdot \left(\E(B_i^2 |X)\right)^{1/2},\\
\E\left(\left.\sum_{i=1}^n A_i B_i\right|X\right)&\le \sum_{i=1}^n \left(\E(A_i^2 |X)\right)^{1/2} \cdot \left(\E(B_i^2 |X)\right)^{1/2}.\end{align*}
Now writing $A_i=\alpha_i(X) \II[i\notin I(1),\ldots,I(k)]$ and $B_i=\alpha_i(X(k))$, we obtain
\begin{align*}
&\sum_{i=1}^{n} \E(\alpha_i(X(k)) \alpha_i(X) \II[i\notin I(1),\ldots,I(k)]|X)\\
&\le \sum_{i=1}^n \left(\E(\alpha_i(X)^2 \II[i\notin I(1),\ldots,I(k)] |X)\right)^{1/2} \cdot \left(\E(\alpha_i(X(k))^2 |X)\right)^{1/2}\\
&=\left(1-\frac{1}{n}\right)^{k/2} \cdot \sum_{i=1}^n (\alpha_i(X)^2)^{1/2} \cdot \left(\E(\alpha_i(X(k))^2 |X)\right)^{1/2}\\
&\le \left(1-\frac{1}{n}\right)^{k/2}\cdot \sum_{i=1}^n \frac{1}{2} \E\left(\alpha_i(X)^2+\alpha_i(X(k))^2|X\right)\\
&\le \left(1-\frac{1}{n}\right)^{k/2}\cdot \frac{1}{2} \E(a g(X) + b + ag(X(k)) +b |X)
\end{align*}
Substituting this into \eqref{mptheta1}, we obtain
\begin{align*}
&m'(\theta)\le \sum_{k=0}^{\infty} \E\left(\theta e^{\theta f(X)}  \frac{1}{n} \sum_{k=0}^{\infty}\left(1-\frac{1}{n}\right)^{k/2} \frac{1}{2} (a g(X) + b + ag(X(k)) +b)\right)\\
&\le \sum_{k=0}^{\infty} \E\left(\theta e^{\theta f(X)}  \frac{1}{n} \sum_{k=0}^{\infty}\left(1-\frac{1}{n}\right)^{k/2}  (a g(X) + b)\right)\\
&\le \E\left(\theta e^{\theta f(X)}  2(a g(X) + b) \right)=
\E\left(\theta e^{\theta f(X)}  (2a f(X) +2a\E g(X) + 2b) \right)\\
&\le \theta 2a m'(\theta)+\theta \left(2a\E g(X) + 2b\right)m(\theta).
\end{align*}
Here we have used the fact that for $\theta>0$,
\begin{equation}\label{eqnfxfxk}
\E(e^{\theta f(X)} f(X(k)))\le \E(e^{\theta f(X)} f(X)),
\end{equation}
since using the exchangeability of $f(X)$ and $f(X(k))$,
\begin{align*}
&\E\left(e^{\theta f(X)} \left(f(X)-f(X(k))\right)\right)=\E\left(e^{\theta f(X(k))} (f(X(k))-f(X))\right)\\
&=\E\left(\left(e^{\theta f(X)}-e^{\theta f(X(k))}\right) \left(f(X)-f(X(k))\right)\right)\ge 0,
\end{align*}
since $e^{\theta f(X)}-e^{\theta f(X(k))}$ and $f(X)-f(X(k))$ always have the same sign. We conclude by applying Lemma \ref{mdlemma}.
\end{proof}
\begin{proof}[Proof of Part 3 of Theorem \ref{thmsbdob} and Corollary \ref{corsbdob} assuming independence]\hspace{5mm}\\
By Lemma \ref{lemmadob2}, 
\[m'(\theta)\ge - \sum_{k=0}^{\infty} \E\left(\left(e^{-\theta}-1\right) e^{\theta f(X)} \cdot  \frac{1}{n} \sum_{i=1}^{n} \alpha_i(X(k)) \alpha_i(X) \II[i\notin I(1),\ldots,I(k)] \right).\]
In Part 2, we proved that
\begin{align*}
&\sum_{i=1}^{n} \E(\alpha_i(X(k)) \alpha_i(X) \II[i\notin I(1),\ldots,I(k)]|X)\\
&\le\left(1-\frac{1}{n}\right)^{k/2}\cdot \frac{1}{2} \E(a g(X) + b + ag(X(k)) +b |X),
\end{align*}
so we obtain
\begin{align}\label{mpboundneg}
&m'(\theta)\ge - \E\left(\left(e^{-\theta}-1\right) e^{\theta f(X)}\frac{1}{n}\right.\\
\nonumber&\left.\cdot \sum_{k=0}^{\infty} \left(1-\frac{1}{n}\right)^{k/2}\cdot \frac{1}{2} \left(a f(X) + a f(X(k)) +2b+2a\E g(X)\right)\right).
\end{align}
The terms involving  $f(X(k))$ cause some difficulty. Although we can show, in the same way as in Part 2, that
\[-\E (e^{\theta f(X)} f(X(k)))  \le -\E (e^{\theta f(X)} f(X)),\]
for us the other sided inequality would be more convenient. Nevertheless, we can use the concentration properties of $f(X(k))$ from Part 2 to bound this term. By Lemma \ref{lemmabrut}, for any $L>0$,
\[\E(e^{\theta f(X)} f(X(k)))\le  L^{-1}\log \E(e^{Lf(X(k))}) m(\theta)+L^{-1}\theta m'(\theta)\]
Now by exchangeability $\E(e^{Lf(X(k))})=\E(e^{Lf(X)})=m(L)$, and we can use the bound from Part 2 to obtain that for $0<L< 1/(2a)$,
\begin{align*}
&\log(m(L))\le \frac{(a\E g(X)+b) L^2 }{(1-2a L)}\\
&\E( e^{\theta f(X)} f(X(k)))\le  \frac{(a\E g(X)+b) L }{(1-2a L)}m(\theta)+L^{-1}\theta m'(\theta)\\
&=\E\left[\frac{(a\E g(X)+b) L }{(1-2a L)}e^{\theta f(X)}+L^{-1}\theta f(X)e^{\theta f(X)}\right]
\end{align*}
Substituting this back to \eqref{mpboundneg}, and summing up in $k$ as previously, we obtain
\begin{align*}
&m'(\theta)\ge - \left(e^{-\theta}-1\right) \\
&\cdot \E\left[ e^{\theta f(X)}  \left(2a\E g(X)+2b+a\frac{(a\E g(X)+b) L }{(1-2a L)}\right)+f(X) e^{\theta f(X)}\left(a+aL^{-1}\theta \right) \right]
\end{align*}
A convenient choice for $L$, which makes the inequality tractable, is $L=-\theta$. With this choice, for $0>\theta>-\frac{1}{2a}$, we obtain
\begin{align*}
&m'(\theta)\ge - \left(e^{-\theta}-1\right)  \left(2a\E g(X)+2b-a\frac{(a\E g(X)+b) \theta }{(1+2a \theta)}\right) m(\theta)\\
&\log(m(\theta))'\ge - \left(e^{-\theta}-1\right)  \left(2a\E g(X)+2b-a\frac{(a\E g(X)+b) \theta }{(1+2a \theta)}\right),
\end{align*}
thus we have shown \eqref{mddobres}.  Now we turn to the proof of the concentration bounds of Corollary \ref{corsbdob}. Suppose that $0>\theta>-\frac{1}{4a}$, then $1+2a \theta\ge 1/2$, so
\begin{equation}\label{eqnmdtheta}
\log(m(\theta))'\ge - \left(e^{-\theta}-1\right)(2-2a\theta)(a\E g(X)+b)
\end{equation}
Now we consider two cases, depending on the size of $a$. The function $\left(e^{x}-1\right)/x$ is increasing for positive $x$, so we can write
\begin{align*}
&-\left(e^{-\theta}-1\right)(2-2a\theta)\ge \frac{\left(e^{\frac{1}{4a}}-1\right)}{1/(4a)}\frac{5}{2}\theta\\
&\log(m(\theta))'\ge \frac{\left(e^{\frac{1}{4a}}-1\right)}{1/(4a)}\frac{5}{2}\theta (a\E g(X)+b)\\
&\log(m(\theta))\le  \frac{\left(e^{\frac{1}{4a}}-1\right)}{1/(4a)}\frac{5}{4}(a\E g(X)+b) \theta^2
\le 2(a\E g(X)+b) \theta^2,
\end{align*}
whenever 
\begin{equation}\label{acriteq}\frac{\left(e^{\frac{1}{4a}}-1\right)}{1/(4a)}\le 
\frac{8}{5},
\end{equation}
that is, whenever $a\ge a_c$ (with $a_c$ defined as in \eqref{aeq}). Using Markov's inequality, we have that for $0<t<\E g(X)$, $0>\theta>-\frac{1}{4a}$,
\[\log \PP(f(X)\le -t)\le \log(m(\theta))+t\theta\le 
2(a\E g(X)+b) \theta^2+\theta t, \]
which takes its minimum at 
\[\theta_{min}=\frac{-t}{4(a\E g(X)+b)},\]
which satisfies $0>\theta>-\frac{1}{4a}$, and thus
\[\log \PP(f(X)\le -t)\le \frac{-t^2}{8(a\E g(X)+b)}.\]
Finally, we need to tackle the case when $a<a_c$. Going back to equation \eqref{eqnmdtheta}, we can write that for $0>\theta>-\frac{1}{4a}$,
\begin{align*}
&\log(m(\theta))'\ge - \left(e^{-\theta}-1\right)\frac{5}{2}(a\E g(X)+b)\\
&\log(m(\theta))\le \left(e^{-\theta}+\theta-1\right)\frac{5}{2}(a\E g(X)+b)
\end{align*}
Let us write $C:=\frac{5}{2}(a\E g(X)+b)$, then by Markov's inequality, we have that for $0>\theta>-\frac{1}{4a}$, $0<t<\E g(X)$,
\[\log(\PP(f(X)\le -t))\le \log(m(\theta))+\theta t\le \left(e^{-\theta}+\theta-1\right)C+\theta t\]
The minimum of the right hand side is taken at 
\[\theta_{min}=-\log\left(1+\frac{t}{C}\right)\ge -\log\left(1+\frac{2}{5}\cdot \frac{1}{a}\right),\]
which satisfies $0>\theta_{min}>-\frac{1}{4a}$ whenever $a<a_c$. Thus, in this case we have
\begin{align*}
&\log(\PP(f(X)\le -t))\le \left(\frac{t}{C}-\log\left(1+\frac{t}{C}\right)\right)C-\log\left(1+\frac{t}{C}\right)t\\
&=C\left[\frac{t}{C}-\log\left(1+\frac{t}{C}\right)\left(1+\frac{t}{C}\right)\right]
\end{align*}
Now let us take a look at the $x-\log(1+x)(1+x)$ function for positive $x$, we can easily check that this is negative, and 
\[x-\log(1+x)(1+x)\le -\frac{x^2}{2+(2/3) x},\]
so
\[\log(\PP(f(X)\le -t))\le -\frac{t^2}{2C+(2/3) t}=-\frac{t^2}{5(a\E g(X)+b)+(2/3) t}.\qedhere\]
\end{proof}

\subsubsection*{Discussion}
When compared to the original proof of Theorem 4.3 of \cite{Cth}, we have introduced several new ideas in the proof. Firstly, instead of bounding \[\Delta(X):=\frac{1}{2}\E(|F(X,X')(f(X)-f(X'))| |X),\] we use the one sided version $(F(X,X'))_+(f(X)-f(X'))_+$. Moreover, we have not taken the expectation of this quantity with respect to $X$, but instead used a tricky symmetrisation argument in \eqref{eqnfxfxk}. Finally, we have also used Lemma \ref{lemmabrut}, which was not needed for the original proof.  In an upcoming paper, we are going to show  that these techniques are powerful enough to imply the exponential and polynomial Efron-Stein inequalities for independent random variables, due to \cite{BoucheronLugosiMassart} and \cite{Lugosimoment}. The dependent case remains an open problem. 

\makeatletter{}\subsection{Dependent case}
In this section, we are going to prove Theorem \ref{thmsbdob} and Corollary \ref{corsbdob}. First, we will clarify the notations in this section. After this, we state two basic lemmas, and a coupling scheme that will be used in the proof. Finally, we give the proof of the results.

Let $X=(X_1,\ldots,X_n)$ be an vector of random variables taking value in $\Lambda$, with Dobrushin interdependence matrix $A=(a_{i,j})_{1\le i,j\le n}$.

Now we will construct a coupling  for $\{X(k)\}_{k\ge 0}$, and $\{X'(k)\}_{k\ge 0}$. Suppose that we have already coupled
\[X(0),\ldots,X(k)\quad\text{and}\quad X'(0),\ldots,X'(k),\]
and that $X(k)=x$, $X'(k)=y$. Let $I(k+1)$ be uniformly chosen from $[n]$, independently of the previously defined variables.
In order to obtain $X_{I(k+1)}(k+1)$ and $X'_{I(k+1)}(k+1)$,
write
\[\nu_1:=\mu_{I(k+1)}(\cdot|x_{-I(k+1)})\quad\text{and}\quad\nu_2:=\mu_{I(k+1)}(\cdot|y_{-I(k+1)}).\]
By Lemma \ref{dtvlemma}, we can define 
the same way as in Section \ref{basicTV}, there exists $B(k+1)$, $C(k+1)$, $D(k+1)$, $\chi(k+1)$  conditionally independent of each other given $X_{-I(k+1)}(k)$ and $X'_{-I(k+1)}(k)$. We can choose $\chi(k+1)\sim \mathrm{Bernoulli}(q)$ for any $q\ge \dtv(\nu_1,\nu_2)$.

Let $\xi(k+1)$ be a random vector taking values in $\{0,1\}^n$, having distribution
\begin{equation}
\xi(k+1):=e_i \text{ with probability } a_{I(k+1),i}\text{ } (i\in [n]), \text{ otherwise }\xi(k+1):=0,
\end{equation}
where $e_i=(0,\ldots,0,1,0,\ldots,0)$ is the $i$th unit vector, and by $0$ we mean the null vector. We suppose that $\xi(k+1)$ is conditionally independent of all else given $I(k+1)$. 
This distribution exists, since 
\[\sum_{i=1}^n a_{I(k+1),i} \le \|A\|_{\infty}\le 1,\]
by our assumptions. Define
\begin{equation}\label{chidefeq}\chi(k+1):=\left<\xi(k+1), L(k)\right>,\end{equation}
with $\left<\cdot,\cdot\right>$ denoting scalar product. Then $\chi(k+1)\sim \mathrm{Bernoulli}(q)$ with 
\[q:=\sum_{i=1}^n a_{I(k+1),i}L_i(k)\ge \dtv(\nu_1,\nu_2).\]
Note that we may have 
$q>\dtv(\nu_1,\nu_2)$, thus our coupling is different from ``the greedy coupling'' that is used on page 76 of \cite{Cth}.

By Lemma \ref{dtvlemma}, we can define \[X_{I(k+1)}(k+1):=(1-\chi(k+1))B(k+1)+\chi(k+1)C(k+1),\] and \[X_{I(k+1)}'(k+1):=(1-\chi(k+1))B(k+1)+\chi(k+1)D(k+1),\]
for all $i\ne I(k+1)$, $X_{i}(k+1):=X_{i}(k)$ and $X'_{i}(k+1):=X'_{i}(k)$.
It is easy to verify by induction that this coupling scheme satisfies Property \/\Propp\/.
For a vector $v\in \R^n$, and $i\in [n]$, define $M(i,v)$ as an $n\times n$ matrix, with 
$(M(i,v))_{l,m}=\II[l=m]$ for every $1\le l,m\le n$ such that $l\ne i$, and $(M(i,v))_{i,m}=v_m$ for every $1\le m\le n$ (thus it equals to the identity matrix in every row except the $i$th one where it equals to $v$). For example,
\[M(3,(1,0,0,0,0))=\left(\begin{array}{ccccc}1 & 0 & 0 & 0 & 0 \\0 & 1 & 0 & 0 & 0 \\1 & 0 & 0 & 0 & 0 \\0 & 0 & 0 & 1 & 0 \\0 & 0 & 0 & 0 & 1\end{array}\right).\]
The following lemma states a recursive bound for $L(k)$.
\begin{lemma}
For the above coupling, for every $k\ge 0$
\begin{equation}\label{eqLkp1Lk}
L(k+1)\le M(I(k+1),\xi(k+1)) L(k),
\end{equation}
and thus
\begin{equation}\label{eqLkp1Lk2}
L(k)\le M(I(k),\xi(k)) \ldots  M(I(1),\xi(1)) L(0).
\end{equation}
\end{lemma}
\begin{proof}
Because of the construction of the coupling, we have $L_i(k)=L_i(k+1)$ if $i\ne I(k+1)$. Moreover, $X_{I(k+1)}(k+1)\ne X_{I(k+1)}'(k+1)$ implies that $\chi(k+1)=1$, so \eqref{eqLkp1Lk} follows by the definitions of $\chi(k+1)$ and  $M(I(k+1),\xi(k+1))$. We obtain \eqref{eqLkp1Lk2} by iteration.
\end{proof}
Note that in Theorem \ref{thmsbdob}, in each of the three cases, $g$ is always going to be bounded, thus $f$ is also bounded. This means that we have $|f(x)|\le C$ for some absolute constant $C$ for every $x\in \Lambda$. Using this and \eqref{eqLkp1Lk2}, we have
\begin{align*}|&\E(f(X(k))-f(X'(k))|X(0)=x,X'(0)=y)|
\\&\quad\le \E(2C \|L(k)\|_1|X(0)=x,X'(0)=y)\le 2C \|[\E(M(I(1),\xi(1)))]^k\|_1\|L(0)\|_1
\\&\quad\le 2nC \left\|\left(1-\frac{1}{n}E + \frac{1}{n}A\right)^k\right\|_1\le 2nC\left(1-\frac{1}{n}+\frac{1}{n}\|A\|_1\right)^k,
\end{align*}
so by summing up, we obtain that \eqref{eq44} holds with $L=2nC/\left(1-\frac{1}{n}+\frac{1}{n}\|A\|_1\right)$. Now we are ready to prove Theorem \ref{thmsbdob} and Corollary \ref{corsbdob}.
\begin{proof}[Proof of Part 1 of Theorem \ref{thmsbdob} and Corollary \ref{corsbdob}]
For $\theta>0$, using Lemma \ref{lemmadob1}, we have
\[m'(\theta)\le \E\left(\sum_{k=0}^{\infty} \left<L(k), \alpha(X(k))\right> \alpha_I(X) \theta e^{\theta f(X)}\right).\]

Let $\{X(k),X'(k)\}_{k\ge 0}$ be defined as in our coupling scheme, then using 
\eqref{eqLkp1Lk2}, and the fact that $L(0)\le e_I$, we can write
\begin{align*}
&\E\left(\left.\left<L(k), \alpha(X(k))\right> \alpha_I(X) \right|X\right)\\
&\quad\le \E\left(\left.\left<\left(M(I(k),\xi(k)) \ldots M(I(1),\xi(1))e_I\right), \alpha(X(k))\right> \alpha_I(X) \right|X\right)\\
&\quad\le \frac{1}{n}\E\left(\left.\alpha(X(k))^{t}\left(M(I(k),\xi(k)) \ldots  M(I(1),\xi(1))\right) \alpha(X) \right|X\right)\\
&\quad\le \frac{1}{n}\E\left(\left.\|\alpha(X(k))\|_{\infty}\left\|M(I(k),\xi(k)) \ldots  M(I(1),\xi(1)) \alpha(X)\right\|_1 \right|X\right).
\end{align*}
Denote by $E$ the identity matrix of size $n$. Using the facts that for *-self-bounding functions, $\|\alpha(X(k))\|_{\infty}\le 1$, and that the elements of $M(I(k),\xi(k))$ and $L(k)$ are non-negative for every $k$, we obtain
\begin{align*}
&\E\left(\left.\left<L(k), \alpha(X(k))\right>\alpha_I(X) \right|X\right)\\
&\quad \le \E\left(\left.\left<M(I(k),\xi(k)) \ldots  M(I(1),\xi(1)) e_{I}, 1\right>\alpha_I(X) \right|X\right),
\end{align*}
with $1$ denoting an $n$ vector of ones. Using the fact that $M(I(1),\xi(1))$, $\ldots$, $M(I(k),\xi(k))$ are independent of $I$ and $X$, we have 
\begin{align*}
&\E\left(\left.\left<L(k), \alpha(X(k))\right>\alpha_I(X) \right|X\right)\\
&\quad\le \frac{1}{n}\left\|\E\left(\left. M(I(k),\xi(k)) \ldots  M(I(1),\xi(1))\right|X\right) \right\|_{1} \|\alpha(X)\|_1\\
&\quad\le \frac{1}{n}\left\|\E\left(\left. M(I(1),\xi(1))\right|X\right)^k \right\|_{1} (ag(X)+b)
\\
&\quad\le\frac{1}{n}\left\|\left(\left(1-\frac{1}{n}\right)E+\frac{1}{n} A \right)^k \right\|_{1} (ag(X)+b)\\
&\quad\le \frac{1}{n}\left(1-\frac{1}{n}+\frac{1}{n}\|A\|_1\right)^k (af(X)+a\E(g)+b),
\end{align*}
We sum up in $k$, and obtain that
\begin{align*}
m'(\theta)&\le \sum_{k=0}^{\infty}  \frac{1}{n}\left(1-\frac{1}{n}+\frac{1}{n}\|A\|_1\right)^k \E\left((af(X)+a\E(g)+b) \theta e^{\theta f(X)}\right),\\
m'(\theta)&\le\frac{1}{1-\|A\|_1} \left(a\theta m'(\theta) +(a\E(g)+b) \theta m(\theta)\right).
\end{align*}
We obtain the mgf bound in Theorem \ref{thmsbdob} by integration of this inequality, and our concentration bound in Corollary \ref{corsbdob} from Lemma \ref{mdlemma}.
\end{proof}
\begin{proof}[Proof of Part 2 of Theorem \ref{thmsbdob} and Corollary \ref{corsbdob}]
As in Part 1, we have that for $\theta>0$,
$m'(\theta)\le \E\left(\sum_{k=0}^{\infty} \left<L(k), \alpha(X(k))\right> \alpha_I(X) \theta e^{\theta f(X)}\right),$
and 
\begin{align*}
&\E\left(\left.\left<L(k), \alpha(X(k))\right> \alpha_I(X) \right|X\right)\\
&\le \frac{1}{n}\E\left(\left.\alpha(X(k))^{t}\left(M(I(k),\xi(k)) \ldots  M(I(1),\xi(1))\right) \alpha(X) \right|X\right)\\
&\le \frac{1}{n}\E\left(\left.\|\alpha(X(k))\|_2\left\|M(I(k),\xi(k)) \ldots  M(I(1),\xi(1)) \alpha(X)\right\|_2 \right|X\right)\\
&\le \frac{1}{n}\E\left(\left.\|\alpha(X(k))\|_2^2\right|X\right)^{1/2}\E\left(\left.\left\|M(I(k),\xi(k)) \ldots  M(I(1),\xi(1)) \alpha(X)\right\|_2^2 \right|X\right)^{1/2}\\
&\le \frac{1}{n}\E\left(\left. ag(X(k))+b\right|X\right)^{1/2} \cdot \E\bigg(\alpha(X)^{t}M(I(1),\xi(1))^{t}\cdot\ldots\cdot M(I(k),\xi(k))^t \\
&\times M(I(k),\xi(k))\cdot \ldots \cdot M(I(1),\xi(1)) \alpha(X)\bigg|X\bigg)^{1/2}\\
&\le \frac{1}{n}\E\left(\left. ag(X(k))+b\right|X\right)^{1/2}\cdot \Bigg(\alpha(X)^{t}\E\bigg(M(I(1),\xi(1))^{t}\cdot\ldots\cdot M(I(k),\xi(k))^t\\
&\times M(I(k),\xi(k))\cdot \ldots \cdot M(I(1),\xi(1)) \bigg|X\bigg)\alpha(X)\Bigg)^{1/2}\\
& \le \frac{1}{n}\E\left(\left. ag(X(k))+b\right|X\right)^{1/2}\left(ag(X)+b\right)^{1/2}\\
&\times \big\|\E\big(M(I(1),\xi(1))^{t}\cdot\ldots\cdot M(I(k),\xi(k))^t 
 M(I(k),\xi(k))\cdot \ldots \cdot M(I(1),\xi(1)) \big|X\big)\big\|_2^{1/2}.
\end{align*}
Now for example
\begin{align*}
&M(3,(1,0,0,0,0))^t \cdot M(3,(1,0,0,0,0))\\
&=\left(\begin{array}{ccccc}1 & 0 & 1 & 0 & 0 \\0 & 1 & 0 & 0 & 0 \\0 & 0 & 0 & 0 & 0 \\0 & 0 & 0 & 1 & 0 \\0 & 0 & 0 & 0 & 1\end{array}\right)\cdot \left(\begin{array}{ccccc}1 & 0 & 0 & 0 & 0 \\0 & 1 & 0 & 0 & 0 \\1 & 0 & 0 & 0 & 0 \\0 & 0 & 0 & 1 & 0 \\0 & 0 & 0 & 0 & 1\end{array}\right) = \left(\begin{array}{ccccc}2 & 0 & 0 & 0 & 0 \\0 & 1 & 0 & 0 & 0 \\0 & 0 & 0 & 0 & 0 \\0 & 0 & 0 & 1 & 0 \\0 & 0 & 0 & 0 & 1\end{array}\right),
\end{align*}
so $M(I(k),\xi(k))^t M(I(k),\xi(k))$ is diagonal, therefore it is easy to see that 
\[M(I(1),\xi(1))^t\ldots M(I(k),\xi(k))^t M(I(k),\xi(k))\ldots M(I(1),\xi(1))\]
is also diagonal. Moreover, by denoting the $n\times n$ matrix of only one 1 at position $i,j$ and zeros elsewhere by $H(i,j)$ and $H(i):=H(i,i)$, we can write
\begin{align*}
&\E(M(I(k),\xi(k))^t M(I(k),\xi(k))|X,I(1),\xi(1),\ldots,I(k-1),\xi(k-1))\\
&=\E(M(I(k),\xi(k))^t M(I(k),\xi(k))|X) \\
&=\frac{1}{n}\sum_{i=1}^n \left[\sum_{j=1}^n a_{i,j}(E-H(i)+H(i,j))^t(E-H(i)+H(i,j))\right.\\
&+\left.\left(1-\sum_{j=1}^n a_{i,j}\right)(E-H(i))^t (E-H(i))\right]\\
&=\frac{1}{n}\sum_{i=1}^n \left[\sum_{j=1}^n a_{i,j}(E-H(i)+H(j))+\left(1-\sum_{j=1}^n a_{i,j}\right)(E-H(i))\right]\\
&=\left(1-\frac{1}{n}\right)E+\frac{1}{n}\sum_{i=1}^n \sum_{j=1}^n a_{i,j} H(j)=\left(1-\frac{1}{n}\right)E+ \frac{1}{n}\sum_{j=1}^n \left(\sum_{i=1}^n a_{i,j}\right) H(j).
\end{align*}
Now using the conditions of our theorem, we have $\left(\sum_{i=1}^n a_{i,j}\right)\le \|A\|_{1}<1$, so we can write 
\begin{align*}
&\E(M(I(k),\xi(k))^t M(I(k),\xi(k))|X,I(1),\xi(1),\ldots,I(k-1),\xi(k-1))\\
&\le \left(1-\frac{1}{n}+\frac{1}{n}\|A\|_1\right)E.
\end{align*}
By repeating this, we obtain that
\begin{align*}
&\left\|\E\left(\left.M(I(1),\xi(1))^{t}\cdot\ldots\cdot M(I(k),\xi(k))^t M(I(k),\xi(k))\cdot \ldots \cdot M(I(1),\xi(1)) \right|X\right)\right\|_2^{1/2}\\
&\le \left(1-\frac{1}{n}+\frac{1}{n}\|A\|_1\right)^{k/2},
\end{align*}
so summing up in $k$, we have
\begin{align*}
&m'(\theta)\\
&\le \frac{1}{n}\E\left(\sum_{k=0}^{\infty} \E\left(\left. ag(X(k))+b\right|X\right)^{1/2}\left(ag(X)+b\right)^{1/2}\right.
\left.\cdot\left(1-\frac{1}{n}+\frac{1}{n}\|A\|_1\right)^{k/2} \theta e^{\theta f(X)}\right)\\
&\le \frac{1}{n}\E\left(\sum_{k=0}^{\infty} \left(\frac{af(X(k))+af(X)+2b+2a\E(g)}{2}\right) \right.\cdot \left.\left(1-\frac{1}{n}+\frac{1}{n}\|A\|_1\right)^{k/2} \theta e^{\theta f(X)}\right)\\
&\le  \frac{1}{n}\E\left(\sum_{k=0}^{\infty} \left(af(X)+b+a\E(g)\right) \left(1-\frac{1}{n}+\frac{1}{n}\|A\|_1\right)^{k/2} \theta e^{\theta f(X)}\right)\\
&\le \E\left(\frac{2}{1-\|A\|_1}\left(af(X)+b+a\E(g)\right) \theta e^{\theta f(X)}\right),
\end{align*}
and the mgf bound in Theorem \ref{thmsbdob} follows by integration.
Here we have used the fact that for $\theta>0$,
\begin{equation}\label{eqnfxfxk}
\E(e^{\theta f(X)} f(X(k)))\le \E(e^{\theta f(X)} f(X)),
\end{equation}
because using the exchangeability of $f(X)$ and $f(X(k))$,
\begin{align*}
&\E\left(e^{\theta f(X)} \left(f(X)-f(X(k))\right)\right)=\E\left(e^{\theta f(X(k))} (f(X(k))-f(X))\right)\\
&=\frac{1}{2}\E\left(\left(e^{\theta f(X)}-e^{\theta f(X(k))}\right) \left(f(X)-f(X(k))\right)\right)\ge 0,
\end{align*}
since $e^{\theta f(X)}-e^{\theta f(X(k))}$ and $f(X)-f(X(k))$ always have the same sign. 
Applying Lemma \ref{mdlemma} with $C=\frac{2a}{1-\|A\|_1}$ and $D=\frac{2(a\E(g)+b)}{1-\|A\|_1}$ proves tail inequality in Corollary \ref{corsbdob}.
\end{proof}

\begin{proof}[Proof of Part 3 of Theorem \ref{thmsbdob} and Corollary \ref{corsbdob}]
Now we will bound the lower tail, so suppose that $\theta<0$. By Lemma \ref{lemmadob2}, 
\[m'(\theta)\ge - \sum_{k=0}^{\infty} \E\left(\left(e^{-\theta}-1\right) e^{\theta f(X)} \left<L(k), \alpha(X(k))\right> \alpha_I \right).\]
In Part 2, we proved that
\begin{align*}
&\E\left(\left.\left<L(k), \alpha(X(k))\right> \alpha_I(X) \right|X\right)\\
&\le \frac{1}{n}\E\left(\left.\frac{af(X(k))+af(X)+2b+2a\E(g)}{2}\right| X\right) \left(1-\frac{1}{n}+\frac{1}{n}\|A\|_1\right)^{k/2}.
\end{align*}
By summing up in $k$, we obtain
\begin{align*}
&m'(\theta)\ge -\left(e^{-\theta}-1\right) \sum_{k=0}^{\infty} \frac{1}{n}\left(1-\frac{1}{n}+\frac{1}{n}\|A\|_1\right)^{k/2}\\
&\times \E\left(\left(\frac{af(X(k))+af(X)+2b+2a\E(g)}{2}\right)e^{\theta f(X)}\right).
\end{align*}
By Lemma \ref{lemmabrut}, since $m(\theta)\ge 1$, for any $L>0$,
\[\E(e^{\theta f(X)} f(X(k)))\le  L^{-1}\log \E(e^{Lf(X(k))}) m(\theta)+L^{-1}\theta m'(\theta),\]
and by Part 2, for $0\le L\le \frac{1-\|A\|_1}{2a}$,
\[\log \E(e^{Lf(X(k))}) =\log(m(L))\le \frac{(a\E(g)+b)L^2}{(1-\|A\|_1-2aL)},\]
so we have
\[\E(e^{\theta f(X)} af(X(k)))\le  a  \frac{(a\E(g)+b)L}{(1-\|A\|_1-2aL)} m(\theta)+a L^{-1}\theta m'(\theta).\]
By the convenient choice of $L=-\theta$, we obtain that for $0\ge \theta\ge -\frac{1-\|A\|_1}{2a}$,
\[\E\left(e^{\theta f(X)} (f(X(k))+f(X))\right)\le  -a \frac{(a\E(g)+b)\theta}{(1-\|A\|_1+2a\theta)} m(\theta),\]
so for $0\ge \theta\ge -\frac{1-\|A\|_1}{2a}$,
\begin{align*}
m'(\theta)&\ge -\left(e^{-\theta}-1\right)  \frac{1}{n}
\sum_{k=0}^{\infty} \left(\frac{-a}{2} \frac{(a\E(g)+b)\theta}{(1-\|A\|_1+2a\theta)}
+a\E(g)+b\right)\\
&\times m(\theta)\left(1-\frac{1}{n}+\frac{1}{n}\|A\|_1\right)^{k/2}\\
&\ge -\left(e^{-\theta}-1\right)\frac{2}{1-\|A\|_1}\left(\frac{-a}{2} \frac{(a\E(g)+b)\theta}{(1-\|A\|_1+2a\theta)}+a\E(g)+b\right)m(\theta),
\end{align*}
which implies \eqref{mddobres}. Suppose that $0\ge \theta\ge -\frac{1-\|A\|_1}{4a}$, then $1-\|A\|_1+2a\theta\ge \frac{1-\|A\|_1}{2}$, so
\begin{equation}\label{dobmdtheta}
m'(\theta)\ge -\left(e^{-\theta}-1\right)\frac{2}{1-\|A\|_1}\left(\frac{1-\|A\|_1-a\theta}{1-\|A\|_1} (a\E(g)+b)\right)m(\theta),
\end{equation}
which implies our mgf bound \eqref{mddobres} in Theorem \ref{thmsbdob}.

We will split the argument for obtaining tail inequalities in Corollary \ref{corsbdob} into into two parts depending on the size of $a$.

First, let $K:=\frac{1-\|A\|_1}{4a}$, then for $0\ge \theta\ge -K$, $\left(e^{-\theta}-1\right)\le
\frac{e^{K}-1}{K}\theta$, and $\frac{1-\|A\|_1-a\theta}{1-\|A\|_1}\le \frac{5}{4}$, so
\begin{align*}
&m'(\theta)\ge -\theta \cdot \frac{e^{K}-1}{K} \frac{1}{1-\|A\|_1} \frac{5}{2} (a\E(g)+b) m(\theta)\\
&\log m(\theta)\le \theta^2 \cdot \frac{e^{K}-1}{K} \frac{1}{1-\|A\|_1} \frac{5}{4} (a\E(g)+b)
\le \frac{2}{1-\|A\|_1} (a\E(g)+b) \theta^2,
\end{align*}
whenever 
\begin{equation}\label{Kcond}\frac{e^{K}-1}{K} \le \frac{8}{5}.
\end{equation}
Let us denote the unique positive solution of the equation
\begin{equation}\frac{e^{x}-1}{x} = \frac{8}{5}
\end{equation}
by $K_{c}$. It is easy to see that $K_c=1/(4a_c)$. For $K\le K_c$, \eqref{Kcond} holds, thus for $a\ge \frac{1-\|A\|_1}{4K_c}=(1-\|A\|_1)a_c$, \eqref{Kcond} holds.
Using Markov's inequality, we obtain that for $0<t<\E(g)$, $0>\theta>-\frac{1-\|A\|_1}{4a}$,
\[\log \PP(f(X)\le -t)\le \log(m(\theta))+t\theta\le  \frac{2}{1-\|A\|_1}(a\E(g)+b)\theta^2+\theta t,\]
which takes its minimum at
\[\theta_{\min}= -\frac{(1-\|A\|_1)t}{4(a\E(g)+b)},\]
which  satisfies $0>\theta_{\min}>-\frac{1-\|A\|_1}{4a}$, and thus
\[\log \PP(f(X)\le -t)\le  -\frac{(1-\|A\|_1)t^2}{8(a\E(g)+b)}.\]
Finally, we need to verify the case when $a<(1-\|A\|_1)a_c$. Going back to equation \eqref{dobmdtheta}, we can write that for $0>\theta>-\frac{1-\|A\|_1}{4a}$,
\begin{align*}
m'(\theta)&\ge -\left(e^{-\theta}-1\right)\frac{2}{1-\|A\|_1}\left(\frac{1-\|A\|_1-a\theta}{1-\|A\|_1} (a\E(g)+b)\right)m(\theta),\\
\nonumber\log(m(\theta))'&\ge - \left(e^{-\theta}-1\right)\frac{5}{2}\frac{1}{1-\|A\|_1}(a\E(g)+b),\\
\nonumber\log(m(\theta))&\le \left(e^{-\theta}+\theta-1\right)\frac{5}{2}\frac{1}{1-\|A\|_1}(a\E(g)+b).
\end{align*}
Let us write $C:=\frac{5}{2}\frac{1}{1-\|A\|_1}(a\E(g)+b)$, then by Markov's inequality, we have that for $0>\theta>-\frac{1-\|A\|_1}{4a}$, $0<t<\E(g)$,
\[\log(\PP(f(X)\le -t))\le \log(m(\theta))+\theta t\le \left(e^{-\theta}+\theta-1\right)C+\theta t\]
The minimum of the right hand side is taken at 
\[\theta_{\min}=-\log\left(1+\frac{t}{C}\right)\ge -\log\left(1+\frac{2}{5}\cdot \frac{1-\|A\|_1}{a}\right),\]
which satisfies $0>\theta_{\min}>-\frac{1-\|A\|_1}{4a}$ whenever $a<a_c(1-\|A\|_1)$. Thus, in this case we have
\begin{align*}
\log(\PP(f(X)\le -t))&\le \left(\frac{t}{C}-\log\left(1+\frac{t}{C}\right)\right)C-\log\left(1+\frac{t}{C}\right)t\\
&=C\left[\frac{t}{C}-\log\left(1+\frac{t}{C}\right)\left(1+\frac{t}{C}\right)\right]
\end{align*}
Now we can verify that the function $x\to x-(1+x)\log(1+x)$ is negative for $x>0$, and 
\[x-(1+x)\log(1+x)\le -\frac{x^2}{2+(2/3) x},\]
so
\[\log(\PP(f(X)\le -t))\le -\frac{t^2}{2C+(2/3) t}=-\frac{t^2}{5(a\E(g)+b)/(1-\|A\|_1)+(2/3) t}.\qedhere\]

\end{proof}

\subsection{The convex distance inequality for dependent random variables}\label{SecProofConvex}
In this section, we prove Theorem \ref{thmTalagrand}. Before turning to the proof, we will state some results. We will use Sion's minimax theorem, which states the following (\cite{Sionmx}, and \cite{Sionel}).
\begin{theorem}
Let $f(x,y)$ denote a function $\mathcal{X}\times \mathcal{Y}\to \R$ that is convex and lower-semicontinuous with respect to $x$, concave and upper-semicontinuous with respect to $y$. If $\mathcal{X}$ is convex and compact, then
\[\inf_{x}\sup_{y} f(x,y)=\sup_{y}\inf_{x}f(x,y)=\min_{x}\sup_{y}f(x,y).\]
\end{theorem}

The following lemma is the $*$-self-bounding analogue of Lemma 1 of \cite{LugosiSelfBounding}.
\begin{lemma}\label{dt2w}
For any $S\in \F$, $d_T^2(x,S)$ is weakly $(4,0)$-$*$-self-bounding, and satisfies that $|d_T^2(x,S)-d_T^2(x^*,S)|\le 1$ for every $x,x^*\in \Lambda$ differing only in one coordinate.
\end{lemma}
\begin{proof}
The second claim is proven in Lemma 1 of \cite{LugosiSelfBounding}.
The proof of the first claim is similar to the proof of Lemma 1 of \cite{LugosiSelfBounding} (see also Proposition 13 of \cite{BoucheronLugosiMassart}).  We recall some  of their argument here. 

Let $\M(S)$ denote the set of probability measures on $S$. Then, using Sion's minimax theorem, we may rewrite $d_T$ as
\begin{equation}\label{dTxS}
d_T(x,S)=\inf_{\nu\in \M(S)}\sup_{\|\alpha\|_2\le 1}\sum_{j=1}^n \alpha_j \E_{\nu}[\II_{x_j\ne Y_j}]
\end{equation}
where $Y=(Y_1,\ldots,Y_n)$ is distributed according to $\nu$. 

We may use once again Sion's minimax theorem to write the convex distance as
\begin{align*}
d_T(x,S)&=\inf_{\nu\in \M(S)}\sup_{\|\alpha\|_2\le 1}\sum_{j=1}^n \alpha_j \E_{\nu}[\II_{x_j\ne Y_j}]\\
&=\sup_{\|\alpha\|_2\le 1}\inf_{\nu\in \M(S)}\sum_{j=1}^n \alpha_j \E_{\nu}[\II_{x_j\ne Y_j}].
\end{align*}
Denote the pair $(\nu,\alpha)$ at which the saddle point is achieved by $(\hat{\nu},\hat{\alpha})$.

Note that strictly speaking, the conditions of Sion's minimax theorem ($\mathcal{X}$ should be convex and compact) are not satisfied, however, this problem can be dealt with the same way as in \cite{BoucheronLugosiMassart} (by mapping the large space $\M(S)$ on the convex compact set of the probability measures on $\{0,1\}^n$).

We can suppose without loss of generality that $d_T^2(y,S)\le d_T^2(x,S)$, thus
\begin{align*}
&d_T^2(x,S)-d_T^2(y,S)=(d_T(x,S)-d_T(y,S))(d_T(x,S)+d_T(y,S))\\
&\le (d_T(x,S)-d_T(y,S))2 d_T(x,S)\le \sum_{i: x_i\ne y_i} 2 d_T(x,S)\hat{\alpha}_i,
\end{align*}
where $\hat{\alpha}_i$  was defined a few lines above. With \[\alpha_i(x):=2 d_T(x,S) \hat{\alpha}_i ,\]
we have
\[\sum_{i=1}^n \alpha_i(x)^2\le 4d_T^2(x,S),\]
so the claim follows. Similarly, analogously to Proposition 13 of \cite{BoucheronLugosiMassart}, one can show that $d_T(x,S)$ is weakly $(1,0)$-$*$-self-bounding.
\end{proof}

Now we are ready to prove the main result of this section.
\begin{proof}[Proof of Theorem \ref{thmTalagrand}]
By Lemma \ref{dt2w}, we can apply Theorem \ref{thmsbdob} to $g(x):=d_T^2(x,S)$ with $a=4$, $b=0$. From \eqref{mddobres}, we obtain for $0\ge \theta\ge -\frac{1-\|A\|_1}{8}$,
\[(\log m(\theta))'\ge -\left(e^{-\theta}-1\right)\frac{2}{1-\|A\|_1}\left(4\E(g)-\theta \frac{8 \E(g)}{(1-\|A\|_1+8\theta)}\right).\]
Here $\left(e^{-\theta}-1\right)\le (-\theta)\frac{e^{1/8}-1}{1/8}$. Let us define $\theta^*:=\frac{\theta}{1-\|A\|_1}$, then the condition $0\ge \theta\ge -\frac{1-\|A\|_1}{8}$ above is equivalent to $0\ge \theta^*\ge -1/8$. Under this assumption, we have
\[(\log m(\theta))'\ge \frac{e^{1/8}-1}{1/8}\theta^*\left(8\E(g)-\theta^* \frac{16 \E(g)}{(1+8\theta^*)}\right).\]
By integration we obtain that
\[\log m(\theta)\le \frac{e^{1/8}-1}{1/8}\E(g) \left(3(\theta^*)^2+\frac{1}{4}\theta^*-\frac{1}{32}\log(1+8\theta^*)\right)(1-\|A\|_1).\]
Now by applying Markov's inequality, we obtain
\begin{align*}
\log[\PP(X\in S)]&=\log[\PP(g(X)-\E(g)\le -\E(g))]\le  m(\theta)+\theta \E(g)\\
&\le\frac{e^{1/8}-1}{1/8}\E(g)\left(3(\theta^*)^2+\frac{1}{4}\theta^*-\frac{1}{32}\log(1+8\theta^*)\right)(1-\|A\|_1)\\
&+(1-\|A\|_1)\theta^* \E(g).
\end{align*}
In order to minimize this, we solve
\[\frac{e^{1/8}-1}{1/8}\theta^*_{m}\left(8\E(g)-\theta^*_{m} \frac{16 \E(g)}{(1+8\theta^*_{m})}\right)=-\E(g),\]
which has solution $\theta^*_{m}\approx -0.0806628>-1/8$, and thus
\begin{align*}
\PP(X\in S) &\le \frac{e^{1/8}-1}{1/8}\E(g) \left(3(\theta^*_{m})^2+\frac{1}{4}\theta^*_{m}-\frac{1}{32}\log(1+8\theta^*_{m})\right)(1-\|A\|_1)\\
&+\theta^*_{m}(1-\|A\|_1)\E(g)\le -\frac{1}{21.345}(1-\|A\|_1) \E(g).
\end{align*}

On the other hand, by \eqref{mgupdob}, we have that for $0\le \theta\le (1-\|A\|_1)/8$,
\[\log \E\left[e^{\theta(g(X)-\E(g))} \right]\le \frac{4\E(g)\theta^2}{(1-\|A\|_1-8\theta)},\]
thus for $\theta=(1-\|A\|_1)/26.1$,
\begin{align*}&\PP(X\in S)\E\left[e^{\theta g(X)}\right]\\
&\quad\le \exp\left(\E(g) \left(\theta+\frac{4\E(g)\theta^2}{(1-\|A\|_1-8\theta)}-\frac{1}{21.345}(1-\|A\|_1)\right)\right)\le 1. \hspace{2mm}\qedhere
\end{align*}
\end{proof}

\makeatletter{}
\section*{Acknowledgements}
The author thanks the anonymous referees for their very valuable, detailed comments. He thanks his thesis supervisors, Louis Chen and Adrian R\"{o}llin for their useful advices.
He thanks Larry Goldstein for his help. He thanks Joel A. Tropp for introducing him to problems related to sampling without replacement. 
He thanks Doma Sz\'{a}sz and Mogyi T\'{o}th for infecting him with their enthusiasm of probability. Finally, many thanks to Roland Paulin for the enlightening discussions.

\bibliographystyle{imsart-nameyear}
\bibliography{References}

\makeatletter{}\section*{Appendix}\label{SecAppendix}
\subsection*{The convex distance inequality for sampling without replacement}
In this section, we first state a version of Talagrand's convex distance inequality for sampling without replacement, and then apply it to the stochastic travelling salesmen problem of Section \ref{SecTSP}.
\begin{theorem}\label{thmsamplingworeplacement}
Let $X=(X_1,\ldots,X_n)$ be a vector of random variables taking values in a set $S=\{A_1,\ldots,A_N\}$. We assume that they are chosen from $S$ without replacement, that is, they are distributed uniformly among the $N\cdot \ldots \cdot (N-n+1)$ possibilities. Let $\Omega:=\{x_1,\ldots, x_n\in S, x_i\ne x_j\text{ for } 1\le i<j\le n \}$, then for any $A\subset \Omega$, we have
\begin{equation}\label{dTXAeqswr}\E(\exp(d_T^2(X,A)/16))\le \frac{1}{\PP(A)},
\end{equation}
with $d_T$ defined as in \eqref{talagranddistance}. Let $g:\Omega\to \R$ be a function satisfying \eqref{nonuccondeq} for some functions $c_i:\Omega\to \R_+$, $1\le i\le n$. Suppose that $\sum_{i=1}^{n} c_i^2(x)\le C$ uniformly in $x\in \Omega$, then for any $t\ge 0$,
\begin{equation}\label{nonunisamplingworeplacementeq}
\PP(|g(X)-\mathbb{M}(g)|\ge t)\le 4\exp\left(\frac{-t^2}{16 C}\right),
\end{equation}
\end{theorem}
\begin{remark}
Note that for sums, Hoeffding and Bernstein-type inequalities for sampling without replacement exist in the literature, see \cite{bardenet2013concentration}.
\end{remark}
This theorem follows from the following result, due to \cite{Talagrand1}.
\begin{theorem}
Denote the symmetric group on $[N]$ by $S_N$, and let $Y:=(Y_1,\ldots, Y_N)$ be distributed uniformly among the $N!$ permutations in $S_N$. Then for any $B\subset S_N$,
\[\E(\exp(d_T^2(Y,B)/16))\le \frac{1}{\PP(B)}.\]
\end{theorem}
\begin{proof}[Proof of Theorem \ref{thmsamplingworeplacement}]
Without loss of generality, assume that $S=[N]$. Let us define $B:=\{x\in S_N: (x_1,\ldots,x_n)\in A\}$. Then it is easy to check that for this choice, for any $x\in S_N$, $d_T(x,B)=d_T((x_1,\ldots,x_n),A)$. This means that 
\begin{align*}&\E[\exp(d_T^2((Y_1,\ldots,Y_n),A)/16)]=\E[\exp(d_T^2(Y,B)/16)]
\\&\quad \le \frac{1}{\PP((X_1,\ldots, X_n)\in B)}=\frac{1}{\PP(A)}.\end{align*}
Now \eqref{dTXAeqswr} follows from the fact that the vectors $(Y_1,\ldots, Y_n)$ and $(X_1,\ldots, X_n)$ have the same distribution.  Finally, we obtain \eqref{nonunisamplingworeplacementeq} similarly to the proof of Lemma 6.2.1 on page 122 of \cite{Steele}.
\end{proof}
As a consequence of these results, we obtain a version of Theorem \ref{thmtspdep} for sampling without replacement.
\begin{theorem}[Stochastic TSP for sampling without replacement]\label{thmtspdepsampwithrepl}
Let $\A=\{a_1,\ldots,a_N\}$ be a set of points in $[0,1]^2$, $X_1,\ldots,X_n$ be sampled without replacement from $\A$, and $T(X_1,\ldots,X_n)$ be the length of the shortest tour according to some cost function $L(x,y)$ satisfying $|x-y|\le L(x,y)\le \mathcal{C}|x-y|$ (as in Section \ref{SecTSP}).
Then for any $t\ge 0$,
\begin{equation}\label{tspconceqswr}\PP(|T(X_1,\ldots,X_n)-\M(T)|\ge t)\le 4\exp\left(-\frac{t^2}{1024\mathcal{C}^2}\right),\end{equation}
where $\M(T)$ denotes the median of $T$.
\end{theorem}
\begin{proof}
This follows from Lemma \ref{spacefillingprop} and \eqref{nonunisamplingworeplacementeq}.
\end{proof}

\end{document}